\documentclass[a4paper,reqno]{amsart}
\usepackage{amssymb, amsmath, amscd}
\usepackage[final]{graphicx}
\usepackage{float}\usepackage{wrapfig}
\usepackage{color}
\usepackage{bbm}

\def\pone{\mathbbm{1}}
\newtheorem{theorem}{Theorem}[section]

\newtheorem{remark}[theorem]{Remark}
\newtheorem{definition}[theorem]{Definition}

\makeatletter
 \@addtoreset{equation}{section}
\makeatother
\begin{document}\vglue -1.5cm
\title
{Applications of PDEs to the study of affine surface geometry}
\author
{P. Gilkey and \, X. Valle-Regueiro}
\thanks{Supported by projects ED431F 2017/03, and MTM2016-75897-P (Spain).}
\address{PBG: Mathematics Department, University of Oregon, Eugene OR 97403-1222, USA}
\email{gilkey@uoregon.edu}
\address{XVR: Facultade de Matem\'aticas, Universidade de Santiago de Compostela,
15782 Santiago de Compostela, Spain}
\email{javier.valle@usc.es}
\thanks{Supported by Project MTM2016-75897-P (AEI/FEDER, UE)}
\begin{abstract}
If $\mathcal{M}=(M,\nabla)$ is an affine surface, let $\mathcal{Q}(\mathcal{M}):=\ker(\mathcal{H}+\frac1{m-1}\rho_s)$ be the
space of solutions to the quasi-Einstein equation for the crucial eigenvalue. 
Let $\tilde{\mathcal{M}}=(M,\tilde\nabla)$ be another affine structure on $M$ which is strongly projectively flat. 
We show that $\mathcal{Q}(\mathcal{M})=\mathcal{Q}(\tilde{\mathcal{M}})$ if and only if
$\nabla=\tilde\nabla$ and that $\mathcal{Q}(\mathcal{M})$ is linearly equivalent to $\mathcal{Q}(\tilde{\mathcal{M}})$ if
and only if $\mathcal{M}$ is linearly equivalent to $\tilde{\mathcal{M}}$. We use these observations to 
classify the flat Type~$\mathcal{A}$ connections up to linear equivalence, to classify the Type~$\mathcal{A}$
connections where the Ricci tensor has rank 1 up to linear equivalence, and to study
the moduli spaces of Type~$\mathcal{A}$ connections where the Ricci tensor is non-degenerate up to affine equivalence.
\end{abstract}
\maketitle

\section{Introduction}
The use of results in the theory of partial differential equations to study geometric questions
is a very classical one. One has, for example, the Hodge-de Rham theorem that the de Rham
cohomology groups of a compact smooth manifold can be identified with the space of harmonic
differential forms; Poincare duality and the Kunneth formula then follow
as does the Bochner vanishing theorem and the fact that the de Rham cohomology of a compact
Lie group can be computed in terms of the cohomology of its Lie algebra.
Applying similar techniques to the spin operator then yields, via the Lichnerowicz formula,
the fact that a compact 4-dimensional spin manifold with non-vanishing first
Pontrjagin class does not admit a metric of positive scalar curvature. One
may use heat equation methods to prove the Riemann-Roch formula for Riemann surfaces using the Dolbeault
operator. There
are many other examples.

Many, but not all, such applications require the manifold be compact and the operator to be
elliptic; in the case of manifolds with boundary one must impose suitable boundary conditions.
And most such applications require the additional structure of a Riemannian metric. By contrast,
in the present paper we will not impose any compactness conditions and we will work in the affine setting
without the structure of an auxiliary Riemannian metric; our analysis is purely local.
In this paper, we will study solutions to the quasi-Einstein
equation in the context of affine geometry. We will focus for the sake of simplicity on affine
homogeneous affine surface geometries of Type~$\mathcal{A}$ (see Definition~\ref{D1.1})
and obtain results concerning the geometry of associated moduli spaces
using purely analytical techniques. Many of these results are new. See, for example, Theorem~\ref{T2.3} where
we show that every Type~$\mathcal{A}$ affine surface geometry is strongly linearly projectively equivalent to a flat 
Type~$\mathcal{A}$
affine surface geometry. In addition, we also derive some previously known
results using analytical techniques that were previously established using techniques of differential geometry.
We hope that the methods introduced here prove useful in other applications to affine geometry.

Let $M$ be a smooth manifold of dimension $m$ which is equipped with a torsion free connection $\nabla$
on the tangent bundle of $M$; the pair $\mathcal{M}=(M,\nabla)$ is called an {\it affine manifold}.
In local coordinates, we adopt the {\it Einstein convention} and
sum over repeated indices to express $\nabla_{\partial_{x^i}}\partial_{x^j}=\Gamma_{ij}{}^k\partial_{x^k}$; the
Christoffel symbols $\Gamma=(\Gamma_{ij}{}^k)$ completely determine the connection. 
Since the connection $\nabla$ is torsion free, one has $\Gamma_{ij}{}^k=\Gamma_{ji}{}^k$.  
Given that we shall, for the most part,
be working only locally, we can assume $M$ is an open subset of $\mathbb{R}^m$ and let the affine structure be
defined by the Christoffel symbols.
If $a$, $b$, $c$, $d$, $e$, and $f$ are real constants, let 
$$\Gamma(a,b,c,d,e,f):=\left\{\begin{array}{lll}\Gamma_{11}{}^1=a,&\Gamma_{11}{}^2=b,&
\Gamma_{12}{}^1=\Gamma_{21}{}^1=c\\
\Gamma_{12}{}^2=\Gamma_{21}{}^2=d,&\Gamma_{22}{}^1=e,&\Gamma_{22}{}^2=f
\end{array}\right\}\,.$$

\begin{definition}\label{D1.1}\rm $M$ is said to be \emph{homogeneous} if for every two points there exists an affine transformation sending one to the other. 
\smallbreak\noindent{\bf Type $\mathcal{A}$:} $\mathcal{M}$ is said to be a {\it Type~$\mathcal{A}$ affine surface geometry}
 if the underlying
manifold is the translation group $M=\mathbb{R}^2$ and if
$\Gamma=\Gamma(a,b,c,d,e,f)$. The group action $(x^1,x^2)\rightarrow(x^1+a^1,x^2+a^2)$ for $(a^1,a^2)\in\mathbb{R}^2$
preserves this geometry so it is homogeneous; the Type~$\mathcal{A}$ connections are the left invariant connections on 
the Lie group $\mathbb{R}^2$.
\smallbreak\noindent{\bf Type $\mathcal{B}$:} $\mathcal{M}$ is said to be a {\it Type $\mathcal{B}$ affine surface geometry}
if the underlying
manifold is the $ax+b$ group $M=\mathbb{R}^+\times\mathbb{R}$ and if $\Gamma=(x^1)^{-1}\Gamma(a,b,c,d,e,f)$.
The group action $(x^1,x^2)\rightarrow(ax^1,ax^2+bx^1)$ for $a>0$ preserves this geometry so it is homogeneous; the
Type~$\mathcal{B}$ connections are the left invariant connections on the $ax+b$ group.
\smallbreak\noindent{\bf Type $\mathcal{C}$:} If $S^2$ is the sphere with the usual round metric and 
if $\nabla$ is the Levi-Civita connection, then $(S^2,\nabla)$ is said to be a {\it Type~$\mathcal{C}$ affine surface geometry}.
 Our notation at this point is a bit non-standard as several of the Type~$\mathcal{B}$ geometries also arise as constant
sectional curvature metrics and thus we have elected not to list these separately as Type~$\mathcal{C}$.
\end{definition}

 If $\mathcal{M}$ is a locally homogeneous affine surface, work of Opozda~\cite{Op04} shows that 
$\mathcal{M}$ is locally affine isomorphic to
one of these 3 geometries; there is a similar classification in the setting of surfaces with torsion which is due
to Arias-Marco and Kowalski~\cite{AMK08}. We also refer to Opozda~\cite{Opozda}, and
Guillot and S\'anchez-Godinez~\cite{G-SG}
for a discussion of more global questions and to Kowalski, Opozda, and Vlasek
\cite{KVOp2,KVOp} for related work. There has been much recent work using this
classification result; we refer, for example, to Derdzinski~\cite{D08},  Dusek~\cite{D09}, and Vanzurova \cite{V13}.

The three classes are not disjoint
as there are geometries which are both Type~$\mathcal{A}$ and Type~$\mathcal{B}$.
In this paper, we shall concentrate on Type~$\mathcal{A}$ structures; in a subsequent paper, we will give a
similar analysis for the Type~$\mathcal{B}$ structures.

Let $\mathcal{M}=(M,\nabla)$ be an affine manifold. If $g\in C^\infty(M)$, we perturb $\nabla$ setting
$
{}^g\nabla_XY:=\nabla_XY+X(g)Y+Y(g)X
$
to define ${}^g\mathcal{M}=(M,{}^g\nabla)$. $\tilde\nabla$ and $\nabla$ are said to be {\it strongly
projectively equivalent} if $\tilde\nabla={}^g\nabla$ for some $g$. In this situation,
$\mathcal{M}=(M,\nabla)$ and $\tilde{\mathcal{M}}=(M,\tilde\nabla)$ are 
said to be {\it strongly projectively equivalent}.
This implies, among other things, that the unparameterized geodesics of $\tilde{\mathcal{M}}$ and 
$\mathcal{M}$ coincide.
We say that $\nabla$ is {\it strongly projectively flat} if $\nabla$ is strongly projectively equivalent to a flat connection;
$\mathcal{M}=(M,\nabla)$ is strongly projectively flat if $\nabla$ is strongly projectively flat.

\begin{theorem}\label{T1.2}
Let $\mathcal{M}$ be a Type~$\mathcal{A}$ affine surface geometry. Then 
$\rho$ is symmetric, $\nabla\rho$ is totally symmetric, and $\mathcal{M}$ is strongly projectively flat.
\end{theorem}

\begin{proof} Let $\Gamma=\Gamma(a,b,c,d,e,f)$ define a Type~$\mathcal{A}$ affine surface geometry.
We show $\rho$ is symmetric and that $\nabla\rho$ is totally symmetric by computing:
\goodbreak\begin{equation}\label{E1.a}\begin{array}{l}
\rho_{11}=-b c + a d - d^2 + b f,\quad\rho_{12}=\rho_{21}=c d - b e,\\
\rho_{22}=-c^2 + a e - d e + c f\end{array}\end{equation}
and
\medbreak\qquad\qquad\qquad$\rho_{11;1}=2 a b c - 2 a^2 d - 2 b c d + 2 a d^2 + 2 b^2 e - 2 a b f$,
\smallbreak\qquad\qquad\qquad$\rho_{11;2}=\rho_{12;1}=\rho_{21;1}=2 b c^2 - 2 a c d + 2 b d e - 2 b c f$,
\smallbreak\qquad\qquad\qquad$\rho_{12;2}=\rho_{21;2}=\rho_{22;1}=2 b c e - 2 a d e + 2 d^2 e - 2 c d f$,
\smallbreak\qquad\qquad\qquad$\rho_{22;2}=-2 c d e + 2 b e^2 + 2 c^2 f - 2 a e f + 2 d e f - 2 c f^2$.
\medbreak\noindent An affine surface $\mathcal{M}$ is strongly projectively flat if and only if both $\rho$ and $\nabla\rho$
are totally symmetric (see, for example \cite{E70, NS} and also \cite{S95} for related results). Thus $\mathcal{M}$ is strongly projectively flat.
\end{proof}

The \emph{affine quasi-Einstein equation} will play a central role in our investigation.
Let $\mathcal{H}=\mathcal{H}_\nabla$ be the {\it Hessian} of an affine manifold:
\smallbreak\centerline{
$\mathcal{H}_\nabla f=(\partial_{x^i}\partial_{x^j}f-\Gamma_{ij}{}^k\partial_{x^k}f)dx^i\otimes dx^j\in S^2(M)$.}
\smallbreak\noindent
 Let $\rho=\rho_{\nabla}$ be the Ricci tensor of $\nabla$ and $\rho_s=\rho_{\nabla,s}$ the associated symmetric Ricci tensor.
 The {\it affine quasi-Einstein operator} (see, for example, \cite{BGGV18})
 is the linear second order partial differential operator
 $\mathcal{H}_\nabla f-\mu f\rho_{\nabla,s}:C^\infty(M)\rightarrow C^\infty(S^2 (M))$.
 The eigenvalue $\mu=-\frac1{m-1}$, where $m=\operatorname{dim}{M}$, plays a distinguished role. We set
\smallbreak\noindent\centerline{
 $\mathcal{Q}(\mathcal{M}):=\{f\in C^\infty(M):\mathcal{H}_\nabla f+\textstyle\frac1{m-1}f\rho_{\nabla,s}=0\}$.}
\smallbreak\noindent We refer to \cite{BGGV18} for the proof of the following results.

\begin{theorem}\label{T1.3}
Let $\mathcal{M}$ be a connected affine manifold of dimension $m$. 
\begin{enumerate}
\item  $\mathcal{Q}({}^g\mathcal{M})=e^g\mathcal{Q}(\mathcal{M})$.
\item If $f\in\mathcal{Q}(\mathcal{M})$ satisfies $f(P)=0$ and $df(P)=0$ for some $P$, then $f\equiv0$. 
\item $\dim\{\mathcal{Q}(\mathcal{M})\}\le m+1$.
If $M$ is simply connected, then $\dim\{\mathcal{Q}(\mathcal{M})\}=m+1$ if and only if $\mathcal{M}$ is strongly
projectively flat.
\end{enumerate}
\end{theorem}

The following result is the major new analytical result of this paper.

\begin{theorem}\label{T1.4}
Let $\mathcal{M}_i:=(M,\nabla_i)$ be  two strongly projectively flat affine structures on the same underlying simply connected manifold $M$ for
$i=1,2$.
Let $\Xi$ be a diffeomorphism of $M$.
\begin{enumerate}
\item If $\mathcal{Q}(\mathcal{M}_1)=\mathcal{Q}(\mathcal{M}_2)$, then $\nabla_1=\nabla_2$.
\item If $\Xi^*\mathcal{Q}(\mathcal{M}_1)=\mathcal{Q}(\mathcal{M}_2)$, then $\Xi^*\nabla_1=\nabla_2$.
\end{enumerate}\end{theorem}

\begin{proof} Since $M$ is simply connected, $\dim\{\mathcal{Q}(\mathcal{M}_i)\}=3$ by Theorem~\ref{T1.3}.
We first establish Assertion~(1) under the stronger assertion that $\nabla_1$ is flat.
Fix $P\in M$. Let $(x^1,\dots,x^m)$ be coordinates on
an open neighborhood $\mathcal{O}$ of $P$ so that all the Christoffel symbols of $\nabla_1$ vanish. This
implies that $\rho_{\nabla_1}=0$ and consequently
$\mathcal{H}_{\nabla_1}f+\frac1{m-1}\rho_{\nabla_1,s}f=\partial_{x^i}\partial_{x^j}f$.
Let $\pone$ be the function which is identically 1.
We have $\{\pone,x^1,\dots,x^m\}\subset\mathcal{Q}(\mathcal{O},\nabla_1)$. 
Because $\pone\in\mathcal{Q}(\mathcal{O},\nabla_1)=\mathcal{Q}(\mathcal{O},\nabla_2)$,
we have $0=\mathcal{H}_{\nabla_2}\pone+\rho_{\nabla_2,s}=\rho_{\nabla_2,s}$.
Since $x^\ell\in\mathcal{Q}(\mathcal{O},\nabla_1)=\mathcal{Q}(\mathcal{O},\nabla_2)$ and
 $\rho_{\nabla_2,s}=0$,
$$
0=\mathcal{H}_{\nabla_2}x^\ell=
(\partial_{x^i}\partial_{x^j}x^\ell+{}^2\Gamma_{ij}{}^k\partial_{x^k}x^\ell)dx^i\otimes dx^j={}^2\Gamma_{ij}{}^\ell dx^i\otimes dx^j
$$
for any $\ell$. This implies ${}^2\Gamma=0$ so $\nabla_1=\nabla_2$ near $P$. As $P$ was
arbitrary, $\nabla_1=\nabla_2$.

We now turn to the general case and assume only that
$\mathcal{M}_1$ is strongly projectively flat. Choose $g$ so $\tilde\nabla_1:={}^{-g}\nabla_1$ is flat.
Assume $\mathcal{Q}(\mathcal{M}_1)=\mathcal{Q}(\mathcal{M}_2)$.
Let $\tilde\nabla_2:={}^{-g}\nabla_2$. By Theorem~\ref{T1.3}, 
$$
\mathcal{Q}(M,\tilde\nabla_2)=e^{-g}\mathcal{Q}(M,\nabla_2)=e^{-g}\mathcal{Q}(M,\nabla_1)=\mathcal{Q}(M,\tilde\nabla_1)\,.
$$
Since $\tilde\nabla_1$ is flat, $\tilde\nabla_2=\tilde\nabla_1$ and
$\nabla_2={}^g\tilde\nabla_2={}^g\tilde\nabla_1=\nabla_1$.
This proves Assertion~(1). Assertion~(2) follows from Assertion~(1) since
$\Xi^*\mathcal{Q}(M,\nabla)=\mathcal{Q}(M,\Xi^*\nabla)$.
\end{proof}

In the rest of this paper, we present applications of Theorem~\ref{T1.4} in the context of Type~$\mathcal{A}$
surface geometries; by Theorem~\ref{T1.3} all these geometries are strongly projectively flat.
In Section~\ref{S2}, we classify the possible forms of $\mathcal{Q}(\mathcal{M})$
where $\mathcal{M}$ is a Type~$\mathcal{A}$ affine surface geometry. In Section~\ref{S3}, 
we study various moduli spaces of Type~$\mathcal{A}$
surface geometries up to linear equivalence.

 In a subsequent paper \cite{GV18x}, we will use the results of Section~\ref{S2} to determine, up
to linear equivalence, which Type~$\mathcal{A}$ surface geometries are geodesically complete
and to re-derive results of D'Ascanio et al.~\cite{DGP17} using different methods.

\section{Relating $\mathcal{M}$ and $\mathcal{Q}$ for Type~$\mathcal{A}$ affine surface geometries}\label{S2}

Let $\Gamma_0$ be defined by taking 
$\Gamma_{ij}{}^k=0$ for  all $i$, $j$, and $k$. The following is a useful technical observation
which holds quite generally.

\begin{theorem}\label{T2.1}
Let $\mathcal{M}=(M,\nabla)$ be a strongly projectively flat simply connected affine manifold of dimension $m$.
\begin{enumerate}
\item Suppose that $\nabla$ is flat. Let $\{\pone,\phi_1,\dots,\phi_m\}$ be a basis for $\mathcal{Q}(\mathcal{M})$.
Let $\Phi:=\{\phi_1,\dots,\phi_m\}$. Then $\det(d\Phi)\ne0$ and $\Phi^*\Gamma_0=\Gamma$.
\item If $\mathcal{M}$ is a surface, then $\mathcal{Q}(\mathcal{M})\ne e^{g(x^1,x^2)}\operatorname{Span}\{f_1(x^1),f_2(x^1),f_3(x^1)\}$.
\end{enumerate}\end{theorem}

\begin{proof}
Suppose $\nabla$ is flat. Then $\rho_{\nabla,s}=0$
and $\mathcal{H}_{\nabla}\pone=0$ so $\pone\in\mathcal{Q}(M,\nabla)$. 
Let $\{\pone,\phi_1,\dots,\phi_m\}$ be a basis for $\mathcal{Q}(\mathcal{M})$.
Let $\Phi:=\{\phi_1,\dots,\phi_m\}$.
Suppose there exists a point $P\in M$ so that ${\det\{d\Phi(P)\}}=0$.
Then there is a non-trivial dependence relation $a_1d\phi_1(P)+\dots+a_md\phi_m(P)=0$. Let 
$$\phi:=a_0\pone+a_1\phi_1+\dots+a_m\phi_m\text{ for }a_0:=-a_1\phi_1(P)-\dots-a_m\phi_m(P)\,.
$$
Since $\phi(P)=0$ and $d\phi(P)=0$, Theorem~\ref{T1.3} shows that $\phi\equiv0$. This contradicts
the assumption that $\{\pone,\phi_1,\dots,\phi_m\}$ is a basis for $\mathcal{Q}(M,\nabla)$. 
Thus $\det(d\Phi)$ is nowhere vanishing so
$\Phi$ defines a local diffeomorphism from $M$ to $\mathbb{R}^m$. 
Fix a simply connected neighborhood $\mathcal{O}$ of a point $P$ of $M$ and let
 $\{x^1:=\phi^1,\dots,x^m:=\phi^m\}$ be the associated local coordinates on $\mathcal{O}$. Then
$\mathcal{Q}(\mathcal{O},\nabla)=\operatorname{Span}\{\pone,x^1,\dots,x^m\}$. 
By Theorem~\ref{T1.4}, this implies that
$\nabla=\Phi^*(\nabla_0)$, where $\nabla_0$ denotes the flat connection on $\mathbb{R}^m$. This completes the proof of Assertion~(1).

Suppose $(M,\nabla)$ is a strongly projectively flat affine surface with
$$
\mathcal{Q}(M,\nabla)=e^{g(x^1,x^2)}\operatorname{Span}\{f_1(x^1),f_2(x^1),f_3(x^1)\}\,.
$$
We argue for
a contradiction. Since $(M,\nabla)$ is strongly projectively flat, we have $\dim\{\mathcal{Q}(M,\nabla)\}=3$
and the functions $\{f_i\}$ are linearly independent.
Let $\tilde\nabla:={}^{-g}\nabla$;
\begin{eqnarray*}
\mathcal{Q}(M,\tilde\nabla)&=&e^{-g(x^1,x^2)}e^{g(x^1,x^2)}\operatorname{Span}\{f_1(x^1),f_2(x^1),f_3(x^1)\}\\
&=&\operatorname{Span}\{f_1(x^1),f_2(x^1),f_3(x^1)\}\,.
\end{eqnarray*}
Let $f=c_1f_1+c_2f_2+c_3f_3$. Since the $f_i$ do not depend on $x^2$, we may choose $(c_1,c_2,c_3)\ne(0,0,0)$ so that
$f(P)=0$ and $df(P)=0$; thus $f$ vanishes
identically so the functions $\{f_1,f_2,f_3\}$ are not linearly independent which is false.
\end{proof}

\begin{theorem}\label{T2.2}
Let $\mathcal{M}$ be a Type~$\mathcal{A}$ affine surface geometry. Let
$\mathcal{Q}_c:=\mathcal{Q}\otimes_{\mathbb{R}}\mathbb{C}$.
\ \begin{enumerate}
\item There is a basis for $\mathcal{Q}_c$ of functions of the form 
$e^{\alpha_1x^1+\alpha_2x^2}p(x^1,x^2)$ where $p$ is a polynomial of degree at most 2 in $(x^1,x^2)$,
where $(\alpha_1,\alpha_2)\in\mathbb{C}^2$,
and where $e^{\alpha_1x^1+\alpha_2x^2}\in\mathcal{Q}_c$.
\item There exist linear functions $L_i$,
there exists a polynomial $Q$ which is at most quadratic, and there exists a
basis $\mathcal{B}$ for $\mathcal{Q}(\mathcal{M})$ which has one of the following four forms
$$\begin{array}{ll}
\mathcal{B}=\{e^{L_1}\cos(L_2),e^{L_1}\sin(L_2),e^{L_3}\},&
\mathcal{B}=\{e^{L_1},e^{L_2},e^{L_3}\},\\[0.05in]
\mathcal{B}=\{e^{L_1},L_2e^{L_1},e^{L_3}\},&
\mathcal{B}=\{e^{L_1},L_2e^{L_1},Qe^{L_1}\}.
\end{array}$$
\end{enumerate}\end{theorem}

\begin{proof} Since $\mathcal{M}$ is a Type~$\mathcal{A}$ affine surface geometry, the quasi-Einstein operator is
a constant coefficient operator.
Consequently, if $f\in\mathcal{Q}(\mathcal{M})$, then $\partial_{x^i}f\in\mathcal{Q}(\mathcal{M})$. As
$\mathcal{M}$ is strongly projectively flat, $\dim\{\mathcal{Q}(\mathcal{M})\}=3$ by Theorem~\ref{T1.3}.
Decompose $\mathcal{Q}_c=\oplus_{\alpha_1,\alpha_2}\mathcal{Q}_{\alpha_1,\alpha_2}$ into the simultaneous generalized eigenspaces of 
$\partial_{x^1}$ and $\partial_{x^2}$
$$
\mathcal{Q}_{\alpha_1,\alpha_2}:=\{f\in\mathcal{Q}_c:(\partial_{x^1}-\alpha_1)^3f=0\text{ and }(\partial_{x^2}-\alpha_2)^3f=0\}\,.
$$
Let $f(x^1,x^2)=e^{\alpha_1x^1+\alpha_2x^2}\tilde f(x^1,x^2)\in\mathcal{Q}_{\alpha_1,\alpha_2}$. We have
$$
0=(\partial_{x^1}-\alpha_1)^3f=e^{\alpha_1x^1+\alpha_2x^2}\partial_{x^1}^3\tilde f\text{ and }
0=(\partial_{x^2}-\alpha_2)^3f=e^{\alpha_1x^1+\alpha_2x^2}\partial_{x^2}^3\tilde f\,.
$$
Thus $\partial_{x^1}^3\tilde f=0$ and $\partial_{x^2}^3\tilde f=0$. Thus $\tilde f$
is a polynomial of degree at most $2$ and applying $\partial_{x^1}$ and $\partial_{x^2}$
appropriately, we see $e^{\alpha_1x^1+\alpha_2x^2}\in\mathcal{Q}_c$. Assertion~(1) follows.

Suppose $e^{\alpha_1x^1+\alpha_2x^2}\in\mathcal{Q}_c$ where $\alpha_i\in\mathbb{C}-\mathbb{R}$ for some $i$.
Let $L_1=\Re(\alpha_1x^1+\alpha_2x^2)$ and let
$L_2=\Im(\alpha_1x^1+\alpha_2x^2)$. Since the quasi-Einstein equation is real, we may take the real and
imaginary parts to see $e^{L_1}\cos(L_2)\in\mathcal{Q}(\mathcal{M})$ and $e^{L_1}\sin(L_2)\in\mathcal{Q}(\mathcal{M})$.
If $pe^{\alpha_1x^1+\alpha_2x^2}\in\mathcal{Q}_c$ for $p$ a polynomial of degree at least 1, then
$$\{\Re(pe^{\alpha_1x^1+\alpha_2x^2}),\Im(pe^{\alpha_1x^1+\alpha_2x^2}),\Re(e^{\alpha_1x^1+\alpha_2x^2}),\Im(pe^{\alpha_1x^1+\alpha_2x^2})\}$$
are 4 linearly independent elements of $\mathcal{Q}(\mathcal{M})$ which is impossible. If there is an element of the form
$p(x^1,x^2)e^{b_1x^1+b_2x^2}\in\mathcal{Q}(\mathcal{M})$, then for dimensional reasons,
$p$ must have degree 0 and the $b_i$ are real. Consequently,
$\mathcal{B}=\{e^{L_1}\cos(L_2),e^{L_1}\sin(L_2),e^{L_3}\}$. 

We therefore assume all the $\alpha_i$ are real and consequently by Assertion~(1), $\mathcal{Q}(\mathcal{M})$ is spanned
by elements of the form $p(x^1,x^2)e^{L(x^1,x^2)}$ where $L(x^1,x^2)=\alpha_1x^1+\alpha_2x^2$ is a real linear function.
The remaining cases are then
examined similarly.
\end{proof}

We use Theorem~\ref{T2.2} to improve Assertion~(2) of Theorem~\ref{T1.2}. We say that two Type~$\mathcal{A}$
connections $\nabla$ and $\tilde\nabla$ are {\it strongly linearly projectively equivalent} if there exists a linear function
$L$ so that $\tilde\nabla={}^L\nabla$ or equivalently, by Theorems~\ref{T1.3} and \ref{T1.4},
$\mathcal{Q}(\tilde\nabla)=e^L\mathcal{Q}(\nabla)$. If $\Gamma_\nabla=\Gamma(a,b,c,d,e,f)$, then {for $L=a_1 x^1 + a_2 x^2$,}
$$
\Gamma_{{}^L\nabla}=\Gamma(a+2a_1,b,c+a_2,d+a_1,e,f+2a_2)\,.
$$
Consequently, the space of Type~$\mathcal{A}$ connections which are strongly linearly projectively equivalent to $\nabla$
is an affine plane in the parameter space $\mathbb{R}^6$.
\begin{theorem}\label{T2.3}
Every Type~$\mathcal{A}$ affine surface geometry is strongly linearly projectively equivalent to a flat Type~$\mathcal{A}$
affine surface geometry.
\end{theorem}

\begin{proof} Let $\Gamma=\Gamma(a,b,c,d,e,f)$ define an affine surface geometry $\mathcal{M}$.
By Theorem~\ref{T2.2}, $\mathcal{Q}(\mathcal{M})$ contains a exponential function $e^L$ where $L$ is real and linear.
Let $\tilde\nabla={}^{-L}\nabla$ define the affine surface geometry $\tilde{\mathcal{M}}$. Since 
$\mathcal{Q}(\tilde{\mathcal{M}})=e^{-L}\mathcal{Q}(\mathcal{M})$, we have $\pone\in\mathcal{Q}(\tilde{\mathcal{M}})$.
This implies $\rho_{\tilde\nabla,s}=0$. By Theorem~\ref{T1.2}, $\rho_{\tilde\nabla}$ is symmetric and thus $\rho_{\tilde\nabla}=0$.
Since we are in dimension 2, this implies $\tilde\nabla$ is flat.
\end{proof}

We say that two connections are {\it linearly equivalent} if there is an element of $\operatorname{GL}(2,\mathbb{R})$
which intertwines them or, equivalently, they differ by a linear change of coordinates.
Note that linear equivalence preserves geodesic completeness.
Linear equivalence was studied in \cite{BGG18} in some detail. If the Ricci tensor has rank $2$, then linear equivalence and
affine equivalence are equivalent notions for Type~$\mathcal{A}$ surfaces; this is not true if the Ricci tensor is degenerate.
For example, not all flat Type~$\mathcal{A}$ connections are linearly equivalent as we shall see in Theorem~\ref{T3.2}.

We examine the possibilities of Theorem~\ref{T2.2} seriatim in what follows. In Section~\ref{S2.1} we suppose
$\mathcal{Q}$ is spanned by 3 distinct real exponentials, in Section~\ref{S2.2}, we suppose $\mathcal{Q}$
contains a complex exponential, and in Section~\ref{S2.3}, we suppose $\mathcal{Q}$ contains a non-trivial
polynomial times an exponential.

\subsection{Type~$\mathcal{A}$ connections with 3 distinct exponentials}\label{S2.1}
We examine the case when $\mathcal{B}=\{e^{L_1},e^{L_2},e^{L_3}\}$ is a basis for $\mathcal{Q}(\mathcal{M})$.
We define the following connections; the computation of $\mathcal{Q}$, $\rho$, and $\det(\rho)$ is then immediate.
\begin{definition}\label{D2.4}\rm 
\ \begin{enumerate}
\item Let $a_1+a_2\ne1$. Set\newline 
$\Gamma_r^2(a_1,a_2):=\frac
{\Gamma(a_1^2+a_2-1,a_1^2-a_1,a_1a_2,a_1a_2,a_2^2-a_2,a_1+a_2^2-1)}{a_1+a_2-1}$. Then\newline
$\mathcal{Q}=\operatorname{Span}\{e^{x^1},e^{x^2},e^{a_1x^1+a_2x^2}\}$, 
$\displaystyle\rho=\frac1{a_1+a_2-1}\left(\begin{array}{cc}a_1^2-a_1&a_1a_2\\a_1a_2&a_2^2-a_2\end{array}\right)$, and
$\det(\rho)=\frac{a_1a_2}{1-a_1-a_2}$. If $a_1a_2\ne0$, then $\operatorname{Rank}\{\rho\}=2$.
\item For $c\ne-1$, set $\Gamma_2^1(c):=\Gamma(-1,0,c,0,0,1+2c)$. Then\newline
$\mathcal{Q}=\operatorname{Span}\{e^{cx^2},e^{(1+c)x^2},e^{cx^2-x^1}\}$ and
$\rho=(c+c^2)dx^2\otimes dx^2$. If $c\ne0$, then
$\operatorname{Rank}\{\rho\}=1$.
\item Set $\Gamma_2^0:=\Gamma(-1,0,0,0,0,1)$.
Then $\mathcal{Q}=\operatorname{Span}\{\pone,e^{x^2},e^{-x^1}\}$ and $\rho=0$.
\end{enumerate}\end{definition}

\begin{theorem}\label{T2.5}
Let $\mathcal{M}=(\mathcal{O},\Gamma)$ be an affine surface where $\mathcal{O}\subset\mathbb{R}^2$ is open.
If there exist distinct linear functions $L_i$ so $\{e^{L_1},e^{L_2},e^{L_3}\}$ is a basis for $\mathcal{Q}({\mathcal{M}})$, then
$\Gamma$ is Type~$\mathcal{A}$ and $\Gamma$
is linearly equivalent to $\Gamma_r^2(a_1,a_2)$ for $a_1+a_2\ne1$ and $a_1a_2\ne0$, to 
$\Gamma_2^1(c)$ for $c\notin\{-1,0\}$, or to $\Gamma_2^0$.\end{theorem}

\begin{proof}
By Theorem~\ref{T1.4}, $\mathcal{Q}=\mathcal{Q}(\mathcal{M})$ determines $\mathcal{M}$.
Since $\dim\{\mathcal{Q}\}=3$,
$\mathcal{M}$ is strongly projectively flat. 
By Theorem~\ref{T1.3}, $\operatorname{Span}\{dL_1,dL_2,dL_3\}$ is 2-dimensional. We assume the notation
is chosen so $dL_1$ and $dL_2$ are linearly independent. Make a linear change of coordinates to ensure that
$x^1=L_1$ and $x^2=L_2$. Because $\dim\{\mathcal{Q}\}=3$,
$\mathcal{Q}=\operatorname{Span}\{e^{x^1},e^{x^2},e^{a_1x^1+a_2x^2}\}$.
If $a_1+a_2=1$, then $a_1-1=-a_2$ and
$\mathcal{Q}=e^{x^1}\operatorname{Span}\{\pone,e^{x^2-x^1},e^{a_2(x^2-x^1)}\}$.
If we make a linear change of coordinates to replace $x^2-x^1$ by $\tilde x^2$, then this
 contradicts Theorem~\ref{T2.1}. This shows that $a_1+a_2\ne1$.
If $\operatorname{Rank}\{\rho \}=2$, then $\Gamma$ is linearly equivalent to $\Gamma_r^2(a_1,a_2)$.
If $\Gamma$ is flat, then we have $a_1=a_2=0$ and after replacing $x^1$ by $-x^1$
we see that $\Gamma$ is linearly equivalent to $\Gamma_2^0$. 
If $\det(\rho)=0$ but $\rho\ne0$, we
have $a_1a_2=0$ but $(a_1,a_2)\ne(0,0)$ and hence $(a_1,a_2)\in\{(a,0),(0,a)\}$ for $a\ne0$. 
Since $a_1+a_2\ne1$, $a\ne1$ so we can make a suitable of change of coordinates to see that $\Gamma$
is linearly equivalent to $\Gamma_2^1(c)$ for some suitably chosen $c$.
\end{proof}

\subsection{Complex exponentials}\label{S2.2} 
We examine the case when we have a basis for $\mathcal{Q}(\mathcal{M})$ of the form
$\mathcal{B}=\{e^{L_1}\cos(L_2),e^{L_1}\sin(L_2),e^{L_3}\}$.
We define the following connections; the computation of $\mathcal{Q}$, $\rho$, and $\det(\rho)$ is then immediate. 
\begin{definition}\label{D2.6}\rm
\ \begin{enumerate}
\item For $b_1\ne1$, set $\Gamma_c^2(b_1,b_2):=\Gamma(1+b_1,0,b_2,1,\frac{1+b_2^2}{b_1-1},0)$. Then\newline
$\mathcal{Q}=\operatorname{Span}\{e^{x^1}\cos(x^2),e^{x^1}\sin(x^2),e^{b_1x^1+b_2x^2}\}$,
$\rho=\displaystyle\left(\begin{array}{cc}b_1&b_2\\b_2&\frac{b_1+b_2^2}{b_1-1}\end{array}\right)$, and\newline
$\det(\rho)=\frac{b_1^2+b_2^2}{b_1-1}$. If $(b_1,b_2)\ne(0,0)$, then $\operatorname{Rank}\{\rho\}=2$.
\item Set $\Gamma_5^1(c):=\Gamma(1,0,0,0,1+c^2,2c)$. Then $\rho=(1+c^2)dx^2\otimes dx^2$,\newline
$\mathcal{Q}=\operatorname{Span}\{e^{cx^2}\cos(x^2),e^{cx^2}\sin(x^2),e^{x^1}\}$, and
   $\operatorname{Rank}\{\rho\}=1$.
\item Set $\Gamma_5^0:=\Gamma(1,0,0,1,-1,0)=\Gamma_c^2(0,0)$. Then\newline
$\mathcal{Q}=\operatorname{Span}\{\pone,e^{x^1}\cos(x^2),e^{x^1}\sin(x^2)\}$, and $\rho=0$.
\end{enumerate}\end{definition}

\begin{theorem}\label{T2.7}
Let $\mathcal{M}=(\mathcal{O},\Gamma)$ be an affine surface where $\mathcal{O}\subset\mathbb{R}^2$ is open. If
$\{e^{L_1}\cos(L_2),e^{L_1}\sin(L_2),e^{L_3}\}$ is a basis for $\mathcal{Q}(\mathcal{M})$
where $\{L_i\}$ are real linear functions, then $\mathcal{M}$ is Type~$\mathcal{A}$, and
$\Gamma$ is linearly equivalent to $\Gamma_c^2(b_1,b_2)$ where {$b_1\neq 1$ and}
 $(b_1,b_2)\ne(0,0)$, to $\Gamma_5^1(c)$ or to 
$\Gamma_5^0$.
\end{theorem}

\begin{proof}
Suppose $\{e^{L_1}\cos(L_2),e^{L_1}\sin(L_2),e^{L_3}\}$ is a basis for $\mathcal{Q}(\mathcal{M})$.
Since $L_2$ is non-trivial, we can make a linear change of coordinates to assume $L_2=x^2$. 
If $L_1$ is not a multiple of $L_2$, change coordinates to assume $L_1=x^1$ and obtain
$\Gamma_c^2(b_1,b_2)$ by setting $L_3=b_1 x^1+b_2 x^2$. If $(b_1,b_2)=(0,0)$, then
$\mathcal{Q}=\operatorname{Span}\{e^{x^1}\cos(x^2),e^{x^1}\sin(x^2),\pone\}$,
$\mathcal{M}$ is flat, and $\Gamma=\Gamma_5^0$.
On the other hand, if $L_1=cx^2$, then we have that
$\mathcal{Q}=e^{cx^2}\operatorname{Span}\{\cos(x^2),\sin(x^2),e^{L_3}\}$. $L_3$ is not independent of $x^1$
by Theorem~\ref{T2.1}. Make a linear change of coordinates to assume $L_3=x^1$
and obtain $\Gamma_5^1(c)$.
\end{proof}

\subsection{Polynomials}\label{S2.3} We assume finally that there is a basis for $\mathcal{Q}$ either of the form
$\mathcal{B}=\{e^{L_1},L_2e^{L_1},e^{L_3}\}$ or
$\mathcal{B}=\{e^{L_1},L_2e^{L_1},Qe^{L_1}\}$.
We define the following connections; the computation of $\mathcal{Q}$ and $\rho$
is then immediate.
\begin{definition}\label{D2.8}\rm 
\ \begin{enumerate}
\item Set
$\Gamma_p^2(a):=\Gamma(2,0,0,1,a,1)$. Then
$\mathcal{Q}=e^{x^1}\operatorname{Span}\{\pone,x^1-ax^2,e^{x^2}\}$, and
$\rho=dx^1\otimes dx^1+adx^2\otimes dx^2$. If $a\ne0$, then $\operatorname{Rank}\{\rho\}=2$.
\item Set
$\Gamma_q^2(\pm1):=\Gamma(2,0,0,1,\pm1,0)$. Then
$\mathcal{Q}=e^{x^1}\operatorname{Span}\{\pone,x^2,2x^1\pm(x^2)^2\}$,
$\rho=dx^1\otimes dx^1\pm dx^2\otimes dx^2$, and $\operatorname{Rank}\{\rho\}=2$.
\item Set $\Gamma_4^1(c):=\Gamma(0,0,1,0,c,2)$. Then
$\mathcal{Q}=e^{x^2}\operatorname{Span}\{\pone,x^2,c(x^2)^2+2x^1\}$ and
$\rho=dx^2\otimes dx^2$.
\item Set $\Gamma_3^1(c):=\Gamma(0,0,c,0,0,1+2c)$. Then
$\mathcal{Q}=\operatorname{Span}\{e^{cx^2},x^1e^{cx^2},e^{(1+c)x^2}\}$, $\rho=(c+c^2)dx^2\otimes dx^2$. If
$c\ne0$ and $c\ne-1$, then $\operatorname{Rank}\{\rho\}=1$.
\item Set $\Gamma_1^1:=\Gamma(-1,0,1,0,0,2)$. Then 
$\mathcal{Q}=\operatorname{Span}\{e^{-x^1+x^2},e^{x^2},x^2e^{x^2}\}$ and
$\rho=dx^2\otimes dx^2$.
\item Set $\Gamma^0_0:=\Gamma(0,0,0,0,0,0)$. Then $\mathcal{Q}=\operatorname{Span}\{\pone,x^1,x^2\}$ and
$\rho=0$.
\item Set $\Gamma^0_1:=\Gamma(1,0,0,1,0,0)$. Then $\mathcal{Q}=\operatorname{Span}\{\pone,e^{x^1},x^2e^{x^1}\}$
and $\rho=0$.
\item Set $\Gamma^0_3:=\Gamma(0,0,0,0,0,1)$. Then $\mathcal{Q}=\operatorname{Span}\{\pone,x^1,e^{x^2}\}$
and $\rho=0$.
\item Set $\Gamma^0_4:=\Gamma(0,0,0,0,1,0)$. Then $\mathcal{Q}=\operatorname{Span}\{\pone,x^2,(x^2)^2+2x^1\}$
and $\rho=0$.
\end{enumerate}\end{definition}

\begin{theorem}\label{T2.9}
Let $\mathcal{M}=(\mathcal{O},\Gamma)$ be an affine surface where $\mathcal{O}\subset\mathbb{R}^2$ is open.
Let $L_i$ be linear functions and let $Q$ be at most quadratic.\begin{enumerate}
\item If $\{e^{L_1},L_2 e^{L_1},e^{L_3}\}$ is a basis
for $\mathcal{Q}(\mathcal{M})$, then $\mathcal{M}$ is Type~$\mathcal{A}$ and
$\Gamma$ is linearly equivalent either to $\Gamma_p^2(a)$ for $a\ne0$ or to $\Gamma_1^1$
or to $\Gamma=\Gamma_3^1(c)$ for $c\ne0$ and $c\ne1$
or to $\Gamma_1^0$ or to $\Gamma_3^0$.
\item If $\{e^{L_1},L_2e^{L_1},Qe^{L_1}\}$ is a basis for $\mathcal{Q}(\mathcal{M})$, then
$\mathcal{M}$ is Type~$\mathcal{A}$ and
$\Gamma$ is linearly equivalent either to $\Gamma_q^2(\pm1)$ or to $\Gamma_4^1(c)$ or to $\Gamma_0^0$
or to $\Gamma_1^0$ or to $\Gamma_4^0$.
\end{enumerate}
\end{theorem}

\begin{proof} By Theorem~\ref{T2.1}, $\mathcal{Q}$ determines $\mathcal{M}$.
We prove Assertion~(1) as follows. Suppose
$\mathcal{Q}(\mathcal{M})=\operatorname{Span}\{e^{L_1},e^{L_1}L_2,e^{L_3}\}$. 
If $L_1\ne0$, we can make a change of variables to assume $L_1=x^1$. If $L_1$ and $L_3$ are linearly independent,
we can change coordinates to assume as well $L_3=x^1+x^2$ and consequently
$$
\mathcal{Q}=e^{x^1}\operatorname{Span}\{\pone,a_1x^1+a_2x^2,e^{x^2}\}\,.
$$
It then follows by Theorem~\ref{T1.4}
that $a_1\ne0$ and thus we may assume $a_1=1$ to obtain 
$\mathcal{Q}=e^{x^1}\operatorname{Span}\{\pone,x^1+a_2x^2,e^{x^2}\}$ and 
obtain $\Gamma_p^2(a_2)$. If $a_2=0$, 
we obtain $\mathcal{Q}=\operatorname{Span}\{e^{x^1},x^1e^{x^1},e^{x^1+x^2}\}$. We make a linear change of
coordinates to assume $\mathcal{Q}=\operatorname{Span}\{e^{x^2},x^2e^{x^2},e^{x^2-x^1}\}$ and obtain $\Gamma_1^1$.
Assume next $L_3=aL_1$ for $a\ne1$ so 
$\mathcal{Q}(\mathcal{M})=\operatorname{Span}\{e^{x^1},e^{ax^1},(a_1x^1+a_2x^2)e^{x^1}\}$.
By Theorem~\ref{T2.1}, $a_2\ne0$ so after a suitable linear change of coordinates we obtain
$\mathcal{Q}(\mathcal{M})=\operatorname{Span}\{e^{x^1},e^{ax^1},x^2e^{x^1}\}$.
 We make another linear change of coordinates to assume 
$$
\mathcal{Q}(\mathcal{M})=\operatorname{Span}\{e^{cx^2},x^1e^{cx^2},e^{(1+c)x^2}\}
$$
and 
we obtain $\Gamma=\Gamma_3^1(c)$.
We have $\operatorname{Rank}\{\rho\}=1$ for $c\ne0,-1$.  
If $a=0$, then  $\mathcal{Q}(\mathcal{M})=\operatorname{Span}\{e^{x^1},\pone,x^2 e^{x^1}\}$ and we get $\Gamma_1^0$.
Finally, if $L_1=0$ we make a change of variables to assume $\operatorname{Q}(\mathcal{M})=\operatorname{Span}\{\pone,x^1,e^{x^2}\}$
and we obtain $\Gamma=\Gamma_3^0$. This completes the proof of Assertion~(1).

We now establish Assertion~(2). Let $\mathcal{Q}(\mathcal{M})=e^{L_1}\operatorname{Span}\{\pone,L_2,Q\}$.
Set $\tilde\Gamma={}^{-L_1}\Gamma$ to obtain $\mathcal{Q}(\tilde{\mathcal{M}})=\operatorname{Span}\{\pone,L_2,Q\}$. 
If $Q$ is linear, then
$\mathcal{Q}(\tilde{\mathcal{M}})=\operatorname{Span}\{\pone,L_2,L_3\}$. Since $L_2$ and $L_3$ are linearly independent,
$\mathcal{Q}(\tilde{\mathcal{M}})=\operatorname{Span}\{\pone,x^1,x^2\}$ so $\tilde\Gamma=\Gamma_0^0$.
If $L_1=0$, then $\Gamma=\Gamma_0^0$. If $L_1\ne0$, we may choose
coordinates to assume $L_1=x^2$. We then have
$\Gamma={}^{x^2}\Gamma_0=\Gamma(0,0,1,0,0,2)$ and $\Gamma=\Gamma_4^1(0)$. On the other hand, if $Q$ is quadratic, then
$\mathcal{Q}(\tilde{\mathcal{M}})=\operatorname{Span}\{\pone,L_2,Q\}$. Change coordinates to assume $L_2=x^2$.
Because $\partial_{x^1}Q\in\mathcal{Q}(\tilde{\mathcal{M}})$ is a multiple of $x^2$, $(x^1)^2$ does not appear in $Q$.
Since $\partial_{x^2}Q$ is a multiple of $x^2$, $x^1x^2$ does not appear in $Q$. 
Thus we may assume 
$Q= (x^2)^2+a_1x^1+a_2x^2$.
Subtracting a multiple of $x^2$ permits to assume $a_2=0$ so 
$\mathcal{Q}(\tilde{\mathcal{M}})=\operatorname{Span}\{\pone,x^2,(x^2)^2+a_1 x^1\}$. 
Theorem~\ref{T2.1} ensures $a_1\neq0$, so we rescale $x^1$
to get $\mathcal{Q}(\tilde{\mathcal{M}})=\operatorname{Span}\{\pone,x^2,(x^2)^2+2  x^1\}$
and $\tilde\Gamma=\Gamma_4^0$. If $L_1=0$, then $\Gamma=\Gamma_4^0$.
Finally, we assume $L_1\ne0$ and
$
\mathcal{Q}(\mathcal{M})=e^{b_1x^1+b_2x^2}\operatorname{Span}\{\pone,x^2,(x^2)^2+2x^1\}$.
Suppose $b_1=0$. Set $\tilde x^2:=b_2x^2$ so
$
\mathcal{Q}(\mathcal{M})=e^{\tilde x^2}\operatorname{Span}\{\pone,\tilde x^2,(2x^1+b_2^{-2}(\tilde x^2)^2)\}\,. 
$ 
Setting $c=b_2^{-2}\ne0$ yields $\Gamma=\Gamma_4^1(c)$; we obtained $\Gamma_4^1(0)$ previously.
Suppose $b_1\ne0$. Let $b_1=\pm c^2$ and $\tilde x^2=cx^2$
setting $\tilde x^1=b_1x^1+b_2x^2$. We have
\begin{eqnarray*}
&&\mathcal{Q}(\mathcal{M})=e^{\tilde x^1}\operatorname{Span}\{\pone,x^2,(x^2)^2+2b_1^{-1}(\tilde x^1-b_2x^2)\}\\
&=&e^{\tilde x^1}\operatorname{Span}\{\pone,x^2,b_1(x^2)^2+2\tilde x^1\}
=e^{\tilde x^1}\operatorname{Span}\{\pone,\tilde x^2,\pm(\tilde x^2)^2+2\tilde x^1\}\,.
\end{eqnarray*}
Thus $\Gamma=\Gamma(2,0,0,1,\pm1,0)=\Gamma_q^2(\pm1)$.
\end{proof}

\section{Spaces of Type~$\mathcal{A}$ connections}\label{S3}
In this section, we apply the results of Section~\ref{S2} to
study moduli spaces of Type~$\mathcal{A}$ connections
up to linear equivalence. In Section~\ref{S3.1} we study flat connections,
in Section~\ref{S3.2} we study connections where the Ricci tensor has rank 1, and in Section~\ref{S3.3} we study connections
where the Ricci tensor has rank 2.
\subsection{Flat Type~$\mathcal{A}$ connections}\label{S3.1}
We collect the connections of Definitions~\ref{D2.4}, \ref{D2.6}, and \ref{D2.8} which are flat (i.e. $\rho=0$) for the sake of convenience.
\begin{definition}\label{D3.1}\rm
\ \begin{enumerate}
\item  $\Gamma_0^0:=\Gamma(0,0,0,0,0,0)$ and $\mathcal{Q}=\operatorname{Span}\{\pone,x^1,x^2\}$.
\item  $\Gamma_1^0:=\Gamma(1,0,0,1,0,0)$ and $\mathcal{Q}=\operatorname{Span}\{\pone,e^{x^1},x^2e^{x^1}\}$.
\item  $\Gamma_2^0:=\Gamma(-1, 0,0,0,0,1)$ and $\mathcal{Q}=\operatorname{Span}\{\pone,e^{-x^1},e^{x^2}\}$.
\item  $\Gamma_3^0:=\Gamma(0,0,0,0,0,1)$ and $\mathcal{Q}=\operatorname{Span}\{\pone,x^1,e^{x^2}\}$.
\item  $\Gamma_4^0:=\Gamma(0,0,0,0,1,0)$ and $\mathcal{Q}=\operatorname{Span}\{\pone,x^2,(x^2)^2+2x^1\}$.
\item  $\Gamma_5^0:=\Gamma(1,0,0,1,-1,0)$ and $\mathcal{Q}=\operatorname{Span}\{\pone,e^{x^1}\cos(x^2),e^{x^1}\sin(x^2)\}$.
\end{enumerate}\end{definition}

\begin{theorem}\label{T3.2}
If $\Gamma$ is a flat Type~$\mathcal{A}$ connection, then $\Gamma$ is linearly equivalent to $\Gamma_i^0$ for some $0\le i\le 5$.
Furthermore, $\Gamma_i^0$ is not linearly equivalent to $\Gamma_j^0$ for $i\ne j$.
\end{theorem}

\begin{proof} By Theorems~\ref{T1.4}, \ref{T2.2}, \ref{T2.5}, \ref{T2.7}, and \ref{T2.9}, every Type~$\mathcal{A}$ connection is linearly
equivalent to one of the connections given in Definitions~\ref{D2.4}, \ref{D2.6}, or \ref{D2.8}. We have listed the 6 connections
of these definitions where $\rho=0$ and thus if $\Gamma$ is a Type~$\mathcal{A}$ connection which is flat, then
$\Gamma$ is linearly equivalent to one of the $\Gamma_i^0$. By inspection, $\mathcal{Q}(\Gamma_i^0)$ is
not linearly equivalent to $\mathcal{Q}(\Gamma_j^0)$ for $i\ne j$ and thus $\Gamma_i^0$ is not linearly equivalent
to $\Gamma_j^0$ for $i\ne j$.
\end{proof}

We now combine the concepts of strong projective equivalence and linear equivalence. In Theorem~\ref{T2.3},
we showed that every Type~$\mathcal{A}$ affine surface geometry $\mathcal{M}$ is strongly linearly projectively equivalent
to a flat Type~$\mathcal{A}$ affine surface geometry $\tilde{\mathcal{M}}$. The following
result now follows by inspection from the definitions given and from Theorem~\ref{T1.4}; it describes the extent
to which $\tilde{\mathcal{M}}$ is not unique.

\begin{theorem}\label{T3.3}
Let $\Gamma$ be a flat Type~$\mathcal{A}$ connection which is strongly linearly projectively equivalent to $\Gamma_i^0$.
Then one of the following possibilities holds:\begin{enumerate}
\item $\Gamma=\Gamma_i^0$.
\item $i=1$, $\mathcal{Q}(\Gamma)=\operatorname{Span}\{e^{-x^1},\pone,x^2\}$,
and $T(x^1,x^2)=(x^2,-x^1)$ intertwines $\Gamma$
and $\Gamma_3^0$.
\item$i=2$,  $\mathcal{Q}(\Gamma)=\operatorname{Span}\{e^{x^1},\pone,e^{x^2+x^1}\}$, and 
$T(x^1,x^2)=(-x^1,x^1+x^2)$ intertwines $\Gamma$
and $\Gamma_2^0$.
\item $i=2$,  $\mathcal{Q}(\Gamma)=\operatorname{Span}\{e^{-x^2},e^{-x^1-x^2},\pone\}$, and 
$T(x^1,x^2)=(x^2,-x^1-x^2)$ intertwines $\Gamma$
and $\Gamma_2^0$. 
\item $i=3$,  $\mathcal{Q}(\Gamma)=\operatorname{Span}\{e^{-x^2},x^1e^{-x^2},\pone\}$, and
$T(x^1,x^2)=(-x^2,x^1)$ intertwines $\Gamma$
and $\Gamma_1^0$.
\end{enumerate}\end{theorem}

\subsection{Type~$\mathcal{A}$ connections where $\boldsymbol{\operatorname{Rank}\{\rho\}=1}$}\label{S3.2}
We collect the connections of Definitions~\ref{D2.4}, \ref{D2.6}, and \ref{D2.8} where $\operatorname{Rank}\{\rho\}=1$
 for the sake of convenience.
 \begin{definition}\label{D3.4}\rm
\ \begin{enumerate}
 \item $\Gamma_1^1:=\Gamma(-1,0,1,0,0,2)$, $\rho=dx^2\otimes dx^2$, and
 $\mathcal{Q}=\operatorname{Span}\{e^{-x^1+x^2},e^{x^2},x^2e^{x^2}\}$.
 \item $\Gamma_2^1(c):=\Gamma(-1,0,c,0,0,1+2c)$ for $c\notin\{0,-1\}$,
$\rho=(c+c^2)dx^2\otimes dx^2$, and
$\mathcal{Q}=\operatorname{Span}\{e^{cx^2},e^{(1+c)x^2},e^{-x^1+cx^2}\}$.
\item $\Gamma_3^1(c):=\Gamma(0,0,c,0,0,1+2c)$ for $c\notin\{0,-1\}$,
$\rho=(c+c^2)dx^2\otimes dx^2$, and
$\mathcal{Q}=\operatorname{Span}\{e^{cx^2},x^1e^{cx^2},e^{(1+c)x^2}\}$.
\item $\Gamma_4^1(c)=\Gamma(0,0,1,0,c,2)$, $\rho=dx^2\otimes dx^2$, and\newline
$\mathcal{Q}=\operatorname{Span}\{e^{x^2},x^2e^{x^2},(2x^1+c(x^2)^2)e^{x^2}\}$ for all $c$.
\item $\Gamma_5^1(c)=\Gamma(1,0,0,0,1+c^2,2c)$, $\rho=(1+c^2)dx^2\otimes dx^2$, and\newline
$\mathcal{Q}=\operatorname{Span}\{e^{c x^2}\cos(x^2),e^{cx^2}\sin(x^2),e^{x^1}\}$.
\end{enumerate}\end{definition}

The following result is now immediate from the discussion we have given. We refer to \cite{BGG18} for a different
proof which uses the Lie algebra of killing vector fields rather than $\mathcal{Q}$; we have chosen a notation which
is in parallel with that used in \cite{BGG18} for the convenience of the reader.

\begin{theorem}
Let $\Gamma$ be a Type~$\mathcal{A}$ connection with $\operatorname{Rank}\{\rho\}=1$.
\begin{enumerate}
\item $\Gamma$
is linearly equivalent to one of the $\Gamma_i^1(\star)$ given above.
\item $\Gamma_i^1(\star)$ is not linearly equivalent to $\Gamma_j^1(\star)$ for $i\ne j$.
\item $\Gamma_2^1(c)$ is linearly equivalent to $\Gamma_2^1(\tilde c)$ if and only if $c= \tilde c$ or $c=-1-\tilde c$. 
\item $\Gamma_3^1(c)$ is not linearly equivalent to $\Gamma_3^1(\tilde c)$ for $c\ne\tilde c$.
\item $\Gamma_4^1(c)$ is linearly equivalent to 
$\Gamma_4^1(\tilde c)$ if and only if $c=\tilde c$ or $c\ne0$ and $\tilde c\ne0$.
\item $\Gamma_5^1(c)$ is not linearly equivalent to $\Gamma_5^1(\tilde c)$ for $c\ne\tilde c$.
\end{enumerate}\end{theorem}

All flat connections are locally affine isomorphic. 
Let $\mathcal{M}$ be a Type~$\mathcal{A}$ affine surface geometry with $\operatorname{Rank}\{\rho\}=1$.
Choose $X\in T_PM$ so $\rho(X,X)\ne0$ and set 
$$
\alpha_X(\mathcal{M}):=\nabla\rho(X,X;X)^2\cdot\rho(X,X)^{-3}
\text{ and }\epsilon_X(\mathcal{M}):=\operatorname{Sign}\{\rho(X,X)\}=\pm1\,.
$$
We refer to \cite{BGG18} for the proof of the following result:
\begin{theorem}\label{T3.6}
Let $\mathcal{M}$ be a Type~$\mathcal{A}$ affine structure with $\operatorname{Rank}\{\rho_{\mathcal{M}}\}=1$. Then
$\alpha(M):=\alpha_X(\mathcal{M})$ and $\epsilon(\mathcal{M}):=\epsilon_X(\mathcal{M})$ are independent of the choice of $X$.
If $\tilde{\mathcal{M}}$ is another Type~$\mathcal{A}$ affine structure with $\operatorname{Rank}\{\rho_{\tilde{\mathcal{M}}}\}=1$,
then $\tilde{\mathcal{M}}$ is locally affine isomorphic to $\mathcal{M}$ if and only if $\alpha(\tilde{\mathcal{M}})=\alpha(\mathcal{M})$
and $\epsilon(\tilde{\mathcal{M}})=\epsilon(\mathcal{M})$.
\end{theorem}
The moduli space is $(-\infty,0]\dot\cup[0,\infty)$ where $0$ appears in 2 different moduli spaces distinguished by $\epsilon$. 
We apply Equation~(\ref{E1.a}) to see:
\begin{equation}\label{E3.a}\begin{array}{llll}
\alpha(\Gamma_1^1)=16,&\epsilon(\Gamma_1^1)=1,\\
\alpha(\Gamma_2^1(c))=\frac{4(1+2c)^2}{c^2+c}\in(-\infty,0]\cup(16,\infty),&\epsilon(\Gamma_2^1(c))=\operatorname{sign}(c^2+c),\\
\alpha(\Gamma_3^1(c))=\frac{4(1+2c)^2}{c^2+c}\in(-\infty,0]\cup(16,\infty),&\epsilon(\Gamma_3^1(c))=\operatorname{sign}(c^2+c),\\
\alpha(\Gamma_4^1(c))=16,&\epsilon(\Gamma_4^1(c))=1,\\
\alpha(\Gamma_5^1(c))=\frac{16c^2}{1+c^2}\in [0,16),&\epsilon(\Gamma_5^1(c))=1.
\end{array}\end{equation}
The following is an immediate consequence of Definition~\ref{D3.4}, Theorem~\ref{T3.6}, and Equation~(\ref{E3.a}). 
\begin{theorem}
	The following are all possible affine equivalences for the connections of Definition~\ref{D3.4}.
\begin{enumerate}
\item   $\Gamma_1^1$ and $\Gamma_4^1(c)$ are affine equivalent to $\Gamma_4^1(\tilde c)$ for any $c$ and $\tilde c$. 
\item  $\Gamma_i^1(c)$ and  $\Gamma_j^1(\tilde c)$, $i,j\in\{2,3\}$ are affine equivalent for $c=\tilde c$ or $c=-1-\tilde c$. 
\item $\Gamma_5^1(c)$ is affine equivalent to $\Gamma_5^1(\tilde c)$ if and only if $c=\pm\tilde c$.
\end{enumerate}\end{theorem}
\subsection{Type~$\mathcal{A}$ connections where $\boldsymbol{\operatorname{Rank}\{\rho\}=2}$}\label{S3.3}
In the context of Type~$\mathcal{A}$ surface geometries with non-degenerate Ricci tensor, linear equivalence and affine equivalence are
the same concept. This vastly simplifies the analysis. 

\begin{theorem}\label{T3.8}
Let $\mathcal{M}$ and $\tilde{\mathcal{M}}$ be Type~$\mathcal{A}$ surface geometries such that
$\rho$ and $\tilde\rho$ are non-degenerate. Then $\mathcal{M}$ is linearly equivalent to $\tilde{\mathcal{M}}$ if
and only if $\mathcal{M}$ is affinely equivalent to $\tilde{\mathcal{M}}$.\end{theorem}

\begin{remark}\rm
Theorem~\ref{T3.8} fails if the Ricci tensor is permitted to be degenerate. For example, Theorem~\ref{T3.2} gives
Type~$\mathcal{A}$ connections which are flat (and hence affinely equivalent) but not linearly equivalent.
It also follows that the structures $\Gamma_2^1(c)$
and $\Gamma_3^1(c)$ are affinely equivalent but not linearly equivalent.
\end{remark}

\begin{proof} Although this follows from work of \cite{BGG18}, we give a slightly different derivation to keep our
present treatment as self-contained as possible. It is immediate that linear equivalence implies affine equivalence.
Conversely, suppose $\nabla_1$ and $\nabla_2$ are two Type~$\mathcal{A}$ connections on $\mathbb{R}^2$.
Let $T$ be a (local) diffeomorphism of $\mathbb{R}^2$ intertwining the two connections. We must show $T$
is linear; the translations play no role.

If $\mathcal{M}$ is a Type~$\mathcal{A}$ affine surface geometry, let $\mathfrak{K}(\mathcal{M})$ be the Lie algebra
of affine Killing vector fields. If $X=a^1\partial_{x^1}+a^2\partial_{x^2}\in\mathfrak{K}(\mathcal{M})$,
let $\mathcal{L}_X$ be the associated Lie derivative. We have by naturality that $\mathcal{L}_X(\rho_{\mathcal{M}})=0$.
Make a linear change of coordinates to ensure
 $\rho=\varepsilon_1dx^1\otimes dx^1+\varepsilon_2dx^2\otimes dx^2$ where $\varepsilon_i^2=1$.
We compute:
\begin{eqnarray*}
0&=&\mathcal{L}_X(\rho_{\mathcal{M}})(Y,Y)=X\rho_{\mathcal{M}}(Y,Y)
-2\rho_{\mathcal{M}}(\mathcal{L}_XY,Y)\\
&=&X\rho_{\mathcal{M}}(Y,Y)
-2\rho_{\mathcal{M}}([X,Y],Y)\,.
\end{eqnarray*}
If we take $Y=\partial_{x^1}$, we obtain $0=-2\rho_{\mathcal{M}}([X,\partial_{x^1}],\partial_{x^1})=\pm2\partial_{x^1}a^1$.
Consequently $\partial_{x^1}a^1=0$ and similarly $\partial_{x^2}a^2=0$. Thus $X=a^1(x^2)\partial_{x^1}+a^2(x^1)\partial_{x^2}$.
If we take $Y=\partial_{x^1}+\partial_{x^2}$ and argue similarly, we obtain $\partial_{x^2}a^1\pm\partial_{x^1}a^2=0$.
Thus $X=(b^1+cx^2)\partial_{x^1}\pm(b^2+cx^1)\partial_{x^2}$. We suppose $c\ne0$ and argue for a contradiction.
Because $\partial_{x^1}$ and $\partial_{x^2}$ are Killing vector fields, we may suppose without loss of generality that
$X=x^2\partial_{x^1}+\varepsilon x^1\partial_{x^2}$ is a Killing vector field where $\varepsilon=\pm1$.
The affine Killing equations $\mathcal{L}_X\nabla=0$
become $[X,\nabla_YZ]-\nabla_Y[X,Z]-\nabla_{[X,Y]}Z=0$ for all $Y,Z\in C^\infty(TM)$. Letting $Y$ and $Z$ be coordinate vector
fields yields
$$\begin{array}{llll}
-\Gamma_{11}{}^2+2 \Gamma_{12}{}^1\varepsilon=0,&
-\Gamma_{11}{}^1\varepsilon+2\Gamma_{12}{}^2\varepsilon=0,\\
\Gamma_{11}{}^1-\Gamma_{12}{}^2+\Gamma_{22}{}^1\varepsilon=0,&
\Gamma_{11}{}^2-\Gamma_{12}{}^1\varepsilon+\Gamma_{22}{}^2\varepsilon=0,\\
2\Gamma_{12}{}^1-\Gamma_{22}{}^2=0,&
2\Gamma_{12}{}^2-\Gamma_{22}{}^1\varepsilon=0.
\end{array}$$
We solve these equations to see $\Gamma=0$ which is impossible since $\rho$ was assumed non-degenerate.
We conclude therefore $\mathfrak{K}(\mathcal{M})=\operatorname{Span}\{\partial_{x^1},\partial_{x^2}\}$. Suppose
$T$ is an affine diffeomorphism. Since the translations are Type~$\mathcal{A}$ affine diffeomorphisms,
we may assume without loss of generality that $T(0)=0$.
We have $T_*\mathfrak{K}(\mathcal{M})=\mathfrak{K}(\mathcal{M})$. Since
$T_*(\partial_{x^i})=a_i^j\partial_{x^j}$, we have $T$ is linear.
\end{proof}

\begin{definition}\rm Let
$\rho_{v,ij}:=\Gamma_{ik}{}^\ell\Gamma_{j\ell}{}^k$, let
$\psi:=\operatorname{Tr}_\rho\{\rho_v\}=\rho^{ij}\rho_{v,ij}$, and let
$\Psi:=\det(\rho_v)/\det(\rho)$.
\end{definition}
It is clear that
$\psi$ and $\Psi$ are invariant under linear equivalence. Consequently by Theorem~\ref{T3.8}, $\psi$ and $\Psi$ are affine invariants in the
context of  Type~$\mathcal{A}$ geometries where $\rho$ is non-singular.
We refer to \cite{BGG16} for the proof of the following result.

\begin{theorem}
Let $\Gamma$ and $\tilde\Gamma$ be two Type~$\mathcal{A}$ connections such that $\rho_\Gamma$ and $\rho_{\tilde\Gamma}$ are
non-degenerate and have the same signature. 
Then $\Gamma$ and $\tilde\Gamma$ are affine equivalent if and only if $(\psi,\Psi)(\Gamma)=(\psi,\Psi)(\tilde\Gamma)$.
\end{theorem}
We show the image of $(\psi,\Psi)$ below in Figure~\ref{Fig-1}; the region on the far right is the moduli space for
positive definite Ricci tensor, the central region is the moduli space for indefinite Ricci tensor,
and the region on the left the moduli space for negative definite Ricci tensor. The left boundary curve between negative definite
and indefinite Ricci tensors is $\sigma_\ell$ (given in red) and the right boundary curve between indefinite and positive definite
Ricci tensors is $\sigma_r$ (given in blue) where
\begin{eqnarray*}
&&\sigma_\ell(t):=(-4 t^2-{t^{-2}}+2,4 t^4-4 t^2+2),\\
&&\sigma_r(t):=(4 t^2+t^{-2}+2,4 t^4+4 t^2+2)\,.
\end{eqnarray*}
\vglue -.3cm\begin{figure}[H]
\caption{Moduli spaces of Type~$\mathcal{A}$ surfaces with $\det(\rho)\ne0$.}\label{Fig-1}
\vglue -.3cm\includegraphics[height=3.5cm,keepaspectratio=true]{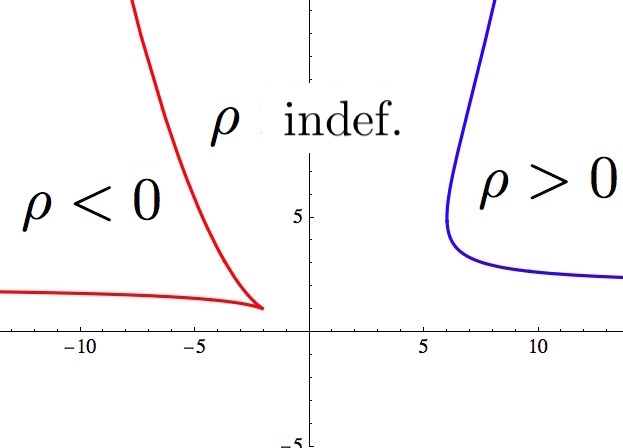}
\end{figure}
\vglue -.3cm\noindent Note that although $(\psi,\Psi)$ is 1-1 on each of the 3 cases separately, the images intersect along the smooth curves
$\sigma_\ell$ and $\sigma_r$. We list below the connections of Section~\ref{S2} where the Ricci tensor has rank $2$ together with the values of $\psi$ and $\Psi$.
\begin{definition}\rm
\ \begin{enumerate}
\item For $a_1+a_2\ne1$ and $a_1a_2\ne0$, set\newline
$\Gamma_r^2(a_1,a_2):=\frac
{\Gamma(a_1^2+a_2-1,a_1^2-a_1,a_1a_2,a_1a_2,a_2^2-a_2,a_1+a_2^2-1)}{a_1+a_2-1}$. Then\newline
$\mathcal{Q}=\operatorname{Span}\{e^{x^1},e^{x^2},e^{a_1x^1+a_2x^2}\}$,
$\rho=\frac1{a_1+a_2-1}\left(\begin{array}{cc}a_1^2-a_1&a_1a_2\\a_1a_2&a_2^2-a_2\end{array}\right)$,
\newline
$\psi=\frac{a_1 - a_1^2 + a_2 + 4 a_1 a_2 + a_1^2 a_2 - a_2^2 + a_1 a_2^2}{a_1a_2}$, and
\newline $\Psi=\frac{1 + a_1 - a_1^2 - a_1^3 + a_2 + 4 a_1 a_2 + a_1^2 a_2 - a_2^2 + a_1 a_2^2 - a_2^3}{a_1a_2}$.
\item For $b_1\ne1$ and $(b_1,b_2)\ne(0,0)$, set $\Gamma_c^2(b_1,b_2):=\Gamma(1+b_1,0,b_2,1,\frac{1+b_2^2}{b_1-1},0)$. Then
$\mathcal{Q}=e^{x^1}\{\cos(x^2),\sin(x^2),e^{(b_1-1)x^1+b_2x^2}\}$,
$\rho=\displaystyle\left(\begin{array}{cc}b_1&b_2\\b_2&\frac{b_1+b_2^2}{b_1-1}\end{array}\right)$,\newline
$\det(\rho)=\frac{b_1^2+b_2^2}{b_1-1}$,
$\psi=\frac{2{b}_1^2+{b}_1^3+6{b}_2^2+4{b}_1+{b}_1{b}_2^2}{{b}_1^2+{b}_2^2}$, and
$\Psi=\frac{2(2+{b}_1^2+3{b}_2^2+2{b}_1+2{b}_1{b}_2^2)}{{b}_1^2+{b}_2^2}$.
\item For $a\ne0$, set
$\Gamma_p^2(a):=\Gamma(2,0,0,1,a,1)$. Then\newline
$\mathcal{Q}=e^{x^1}\operatorname{Span}\{\pone,x^1-ax^2,e^{x^2}\}$, 
$\rho=dx^1\otimes dx^1+adx^2\otimes dx^2$, and \newline$(\psi,\Psi)=(7,10)+\frac1a(1,4)$.
\item Set
$\Gamma_q^2(\pm1):=\Gamma(2,0,0,1,\pm1,0)$. Then
$\mathcal{Q}=e^{x^1}\operatorname{Span}\{\pone,x^2,2x^1\pm(x^2)^2\}$,\newline
$\rho=dx^1\otimes dx^1\pm dx^2\otimes dx^2$, and $(\psi,\Psi)=(7,10)$.
\end{enumerate}\end{definition}

\vglue -.1cm\noindent{\bf Case 1:} Linear equivalence where 
${\mathcal{Q}(\mathcal{M})=\operatorname{Span}\{e^{L_1},e^{L_2},e^{L_3}\}}$.
Suppose that $\{L_i,L_j\}$ are linearly independent for $i\ne j$. Let $\sigma$ be
a permutation of the integers $\{1,2,3\}$. Introduce new coordinates $y_\sigma^1:=L_{\sigma(1)}(x^1,x^2)$ and
$y_\sigma^2:=L_{\sigma(2)}(x^1,x^2)$. Expand $L_{\sigma(3)}(x^1,x^2)=a_{1,\sigma}y_\sigma^1+a_{2,\sigma}y_\sigma^2$
to express 
$$
\mathcal{Q}(\mathcal{M})=\operatorname{Span}\{e^{y_\sigma^1},e^{y_\sigma^2},
e^{a_{1,\sigma}y_\sigma^1+a_{2,\sigma}y_\sigma^2}\}\,.
$$
This structure is defined by the pair $(a_{1,\sigma},a_{2,\sigma})$; there are, generically, 6 such pairs that give rise to the
same affine structure up to linear equivalence. We say $(a_1,a_2)\sim(\tilde a_1,\tilde a_2)$
if $\Gamma_r^2(a_1,a_2)$ is linearly equivalent to $\Gamma_r^2(\tilde a_1,\tilde a_2)$, 
i.e. there exists $T$ in $GL(\mathbb{R}^2)$ so
$T^*\operatorname{Span}\{e^{x^1},e^{x^2},e^{a_1x^1+a_2x^2}\}
=\operatorname{Span}\{e^{\tilde x^1},e^{\tilde x^2},e^{\tilde a_1\tilde x^1+\tilde a_2\tilde x^2}\}$.
Suppose that $L_1=x^1$, $L_2=x^2$, and $L_3=a_1x^1+a_2x^2$.
Let $\sigma_{ijk}$ be the permutation $1\rightarrow i$, $2\rightarrow j$, $3\rightarrow k$. 
We have
$$\begin{array}{llllll}
\sigma_{123}:&y^1=L_1,&y^2=L_2,&L_3=a_1y^1+a_2y^2,&(a_1,a_2)\sim(a_1,a_2).\\[0.05in]
\sigma_{213}:&y^1=L_2,&y^2=L_1,&L_3=a_2y^1+a_1y^2,&({a_1},{a_2})\sim({a_2},{a_1}).\\[0.05in]
\sigma_{132}:&y^1=L_1,&y^2={L_3},&L_2=-\frac{a_1}{a_2}y^1+\frac1{a_2}y^2,&({a_1},{a_2})\sim(-\frac{a_1}{a_2},\frac1{a_2}).\\[0.05in]
\sigma_{321}:&y^1={L_3},&y^2=L_2,&L_1=\frac1{a_1}y^1-\frac{a_2}{a_1}y^2,&({a_1},{a_2})\sim(\frac1{a_1},-\frac{a_2}{a_1}).\\[0.05in]
\sigma_{231}:&y^1=L_2,&y^2={L_3},&L_1=-\frac{a_2}{a_1}y^1+\frac1{a_1}y^2,&(a_1,a_2)\sim(-\frac{a_2}{a_1},\frac1{a_1}).\\[0.05in]
\sigma_{312}:&{y^1= L_3,}&{y^2= L_1},&L_2=\frac1{a_2}y^1-\frac{a_1}{a_2}y^2,&(a_1,a_2)\sim(\frac1{a_2},-\frac{a_1}{a_2}). 
\end{array}$$
We observe that since $\psi$ and $\Psi$ are linear invariants, they are constant under the action of the group of permutations $s_3$.
Although generically $s_3$ acts without fixed points, there are degenerate cases where the action is not fixed point free.

If $\det(\rho)>0$ and $\operatorname{Tr}(\rho)<0$, then $\rho$ is negative definite;
if $\det(\rho)>0$ and $\operatorname{Tr}(\rho)>0$, then $\rho$ is positive definite;
if $\det(\rho)<0$, then $\rho$ is indefinite. The six lines
$\{x=0,\ x=-1,\ y=0,\ y=-1,\ x+y=1,\ x=y\}$
are given in black below; they further divide the regions 
where $\rho$ is negative definite (light blue), $\rho$ is positive definite (yellow), and
$\rho$ is indefinite (green); the three regions in different colors can be further divided  into 6 regions under the action of $s_3$.\vglue -.3cm
\begin{figure}[H]
\caption{The six lines.}\label{Fig-2}
\vglue -.3cm\includegraphics[height=2.5cm,keepaspectratio=true]{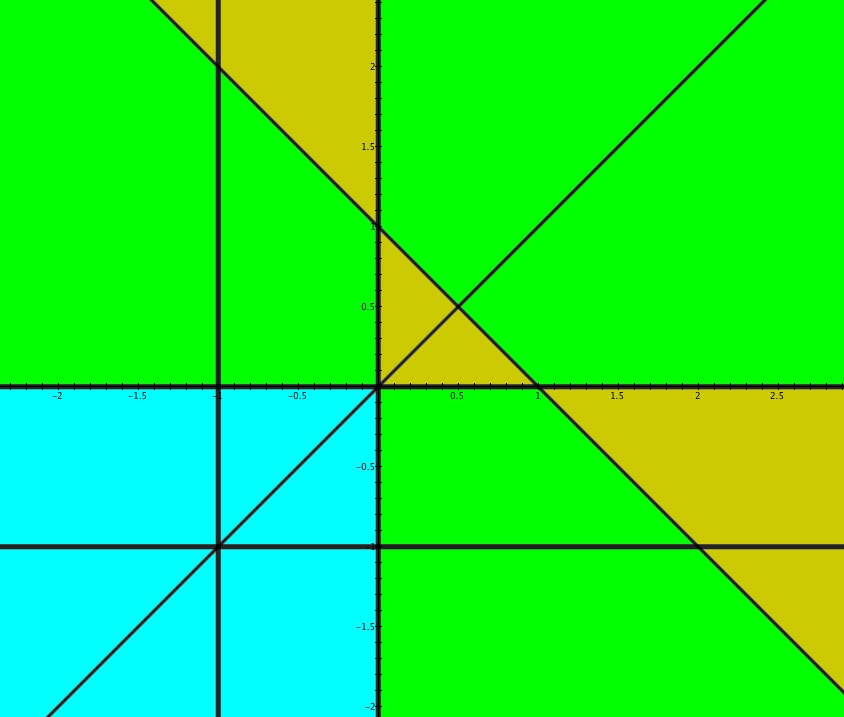}
\end{figure}
\vglue -.5cm\noindent{\bf Case 1a:} The Ricci tensor is negative definite.
A fundamental region for the moduli space where $\rho$ is negative definite is the triangle given by the inequalities
$-1\le y\le x <0$; the other 5 fundamental regions are obtained from this one by applying $s_3$; the regions intersect along the
lines $x=y$, $x=-1$, and $y=-1$. The point $(-1,-1)$ is the singular point which is preserved by $s_3$ which is the maximal symmetry group;
this corresponds to the cusp.  We obtain the full moduli space as every $\Gamma$ where $\rho_\Gamma<0$ is represented by 3 distinct exponentials
which are, up to linear equivalence, $\{e^{x^1},e^{x^2},e^{a_1x^1+a_2x^2}\}$ for $a_1a_2\ne0$, and $1\ne a_1+a_2$. This is not true
in positive definite and indefinite setting as we only obtain a part of the moduli space in these cases. We give
the fundamental domain for $\rho<0$ below in Figure~\ref{Fig-3}, the images under $s_3$, and the image in the moduli space.
The boundary curve $\sigma_\ell$ in the moduli space is the image of the boundary of the open triangle. The
curve $(\psi(t,t),\Psi(t,t))$ for $-1\le t<0$ is given in red and the curve $(\psi(t,-1),\Psi(t,-1))$ for $-1\le t<0$ is given in blue.
These curves are preserved by a $\mathbb{Z}_2$ subgroup of $s_3$.
The final boundary segment $(0,t)$ of the triangle for $0\le t\le -1$ marked in black has no geometric
significance.\goodbreak
\begin{figure}[H]
\caption{The fundamental domains for $\rho<0$.}\label{Fig-3}
\vglue -.3cm
\includegraphics[height=2.5cm,keepaspectratio=true]{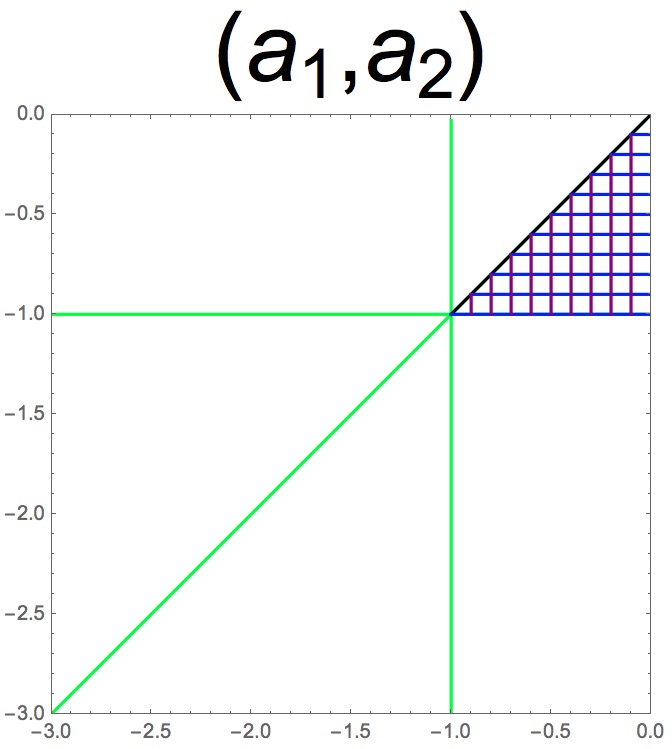}\ 
\includegraphics[height=2.5cm,keepaspectratio=true]{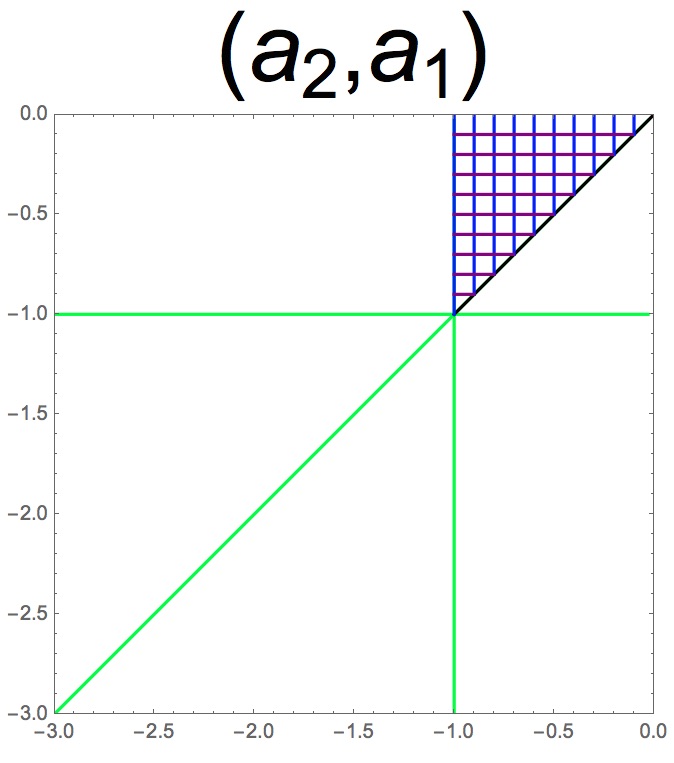}\ 
\includegraphics[height=2.5cm,keepaspectratio=true]{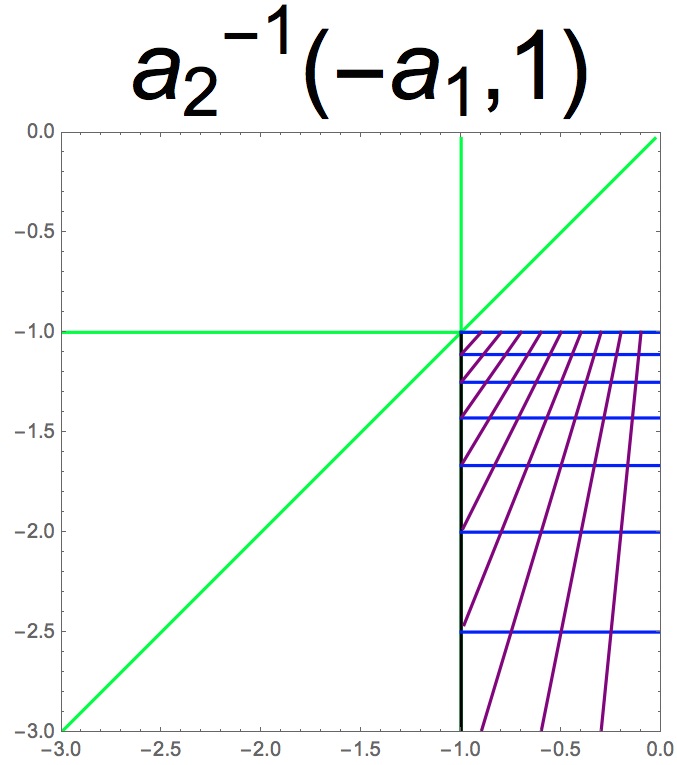}\ 
\includegraphics[height=2.5cm,keepaspectratio=true]{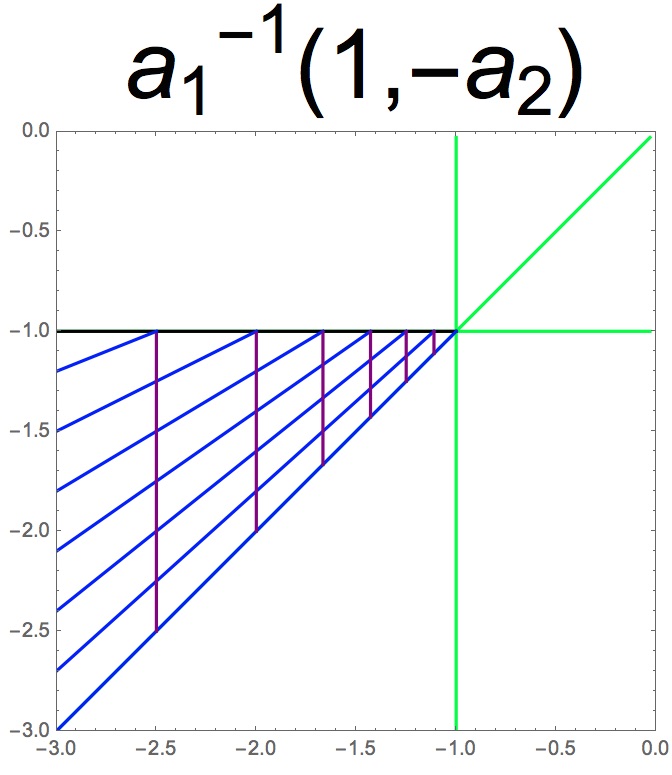}\ 
\includegraphics[height=2.5cm,keepaspectratio=true]{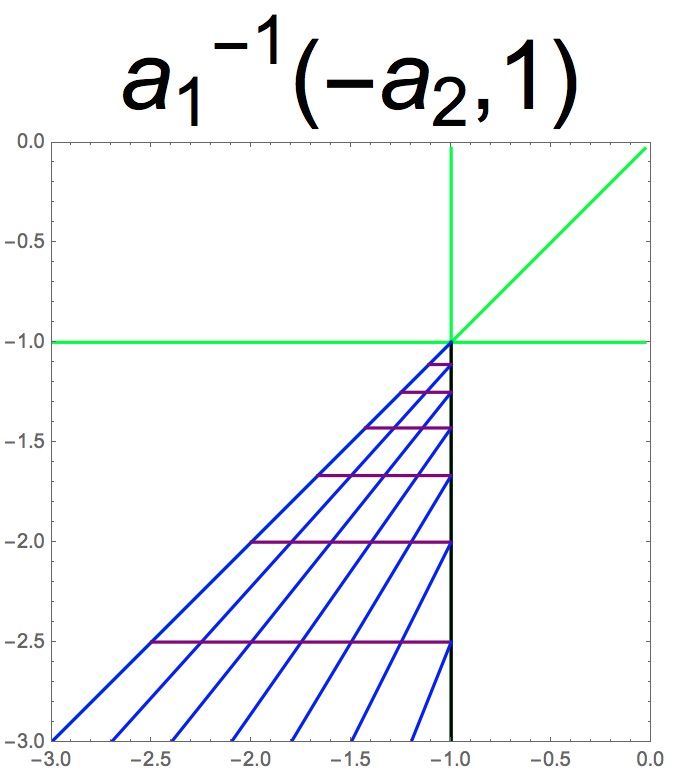}\ 
\includegraphics[height=2.5cm,keepaspectratio=true]{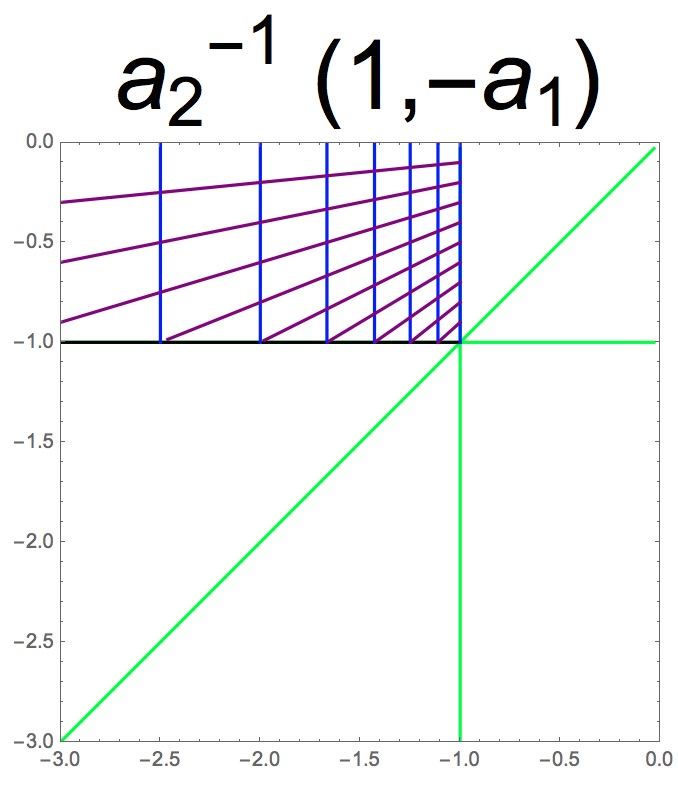}\ 
\includegraphics[height=2.5cm,keepaspectratio=true]{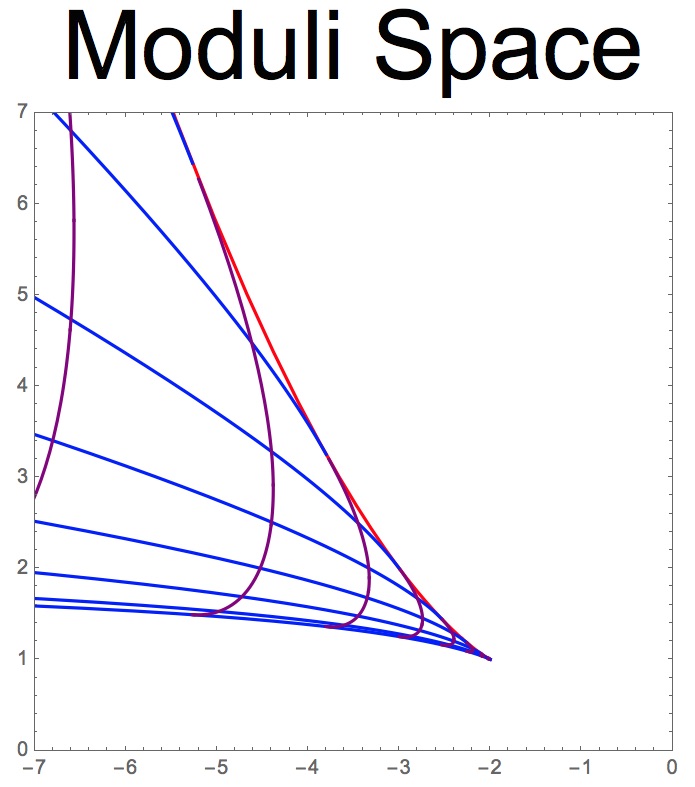}\ 
\includegraphics[height=2.5cm,keepaspectratio=true]{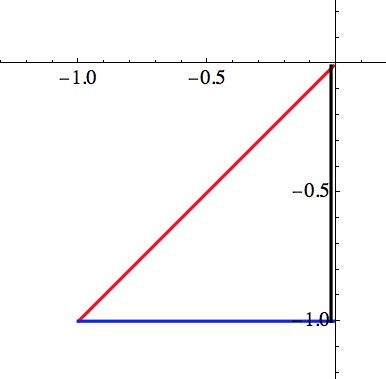}\ 
\includegraphics[height=2.5cm,keepaspectratio=true]{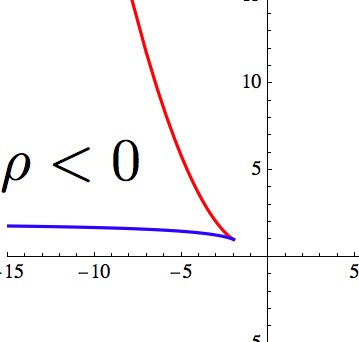}
\end{figure}

\vglue -.3cm\noindent{\bf Case 1b:} The Ricci tensor is indefinite. 
 A fundamental region is given by the inequalities $0<y<x$ and $x+y>1$. There are portions of the moduli space where the Ricci tensor is
indefinite not present in this fundamental region. The region extends indefinitely to the right and to the top; there is no boundary.
Below in Figure~\ref{Fig-4}, we give a fundamental domain and the various images under the symmetric group $s_3$.
The ideal curve $(t,0^-)$ for $t\in(0,1)$ marked in blue maps to the exceptional ray  $(7,10)-t(1,4)$, for $t>0$; this is not in the image
of the moduli space as the exceptional ray arises from the structures where $\mathcal{Q}$ contains
a polynomial as we shall see presently.  The curve $(2t,-1)$ for $t\in(0,1)$ marked in red maps to the
part of the boundary curve $\sigma_r$ which is below the line $\Psi=10$.\vglue -.3cm
\begin{figure}[H]
\caption{The fundamental domains for $\rho$ indefinite.}\label{Fig-4}
 \vglue -.3cm
 \includegraphics[height=3.0cm,keepaspectratio=true]{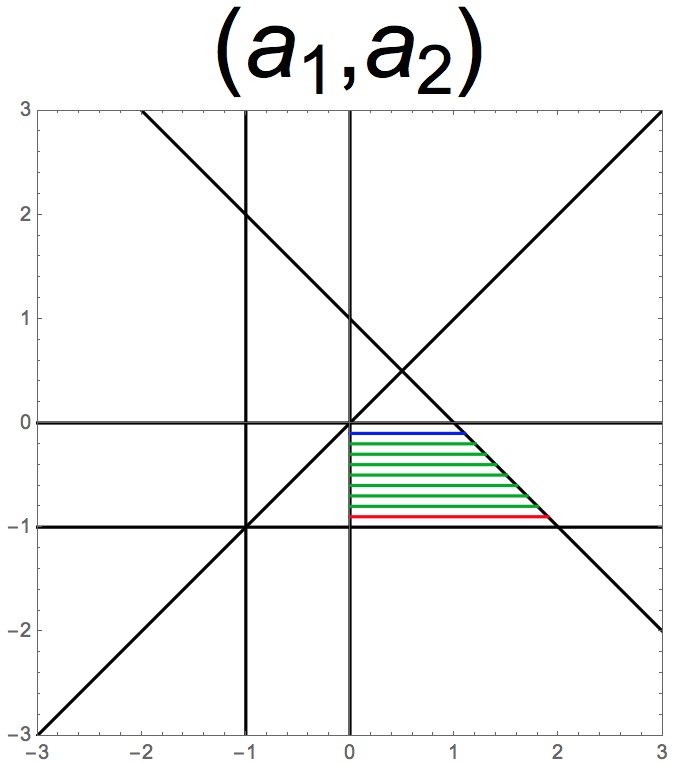}\  
\includegraphics[height=3.0cm,keepaspectratio=true]{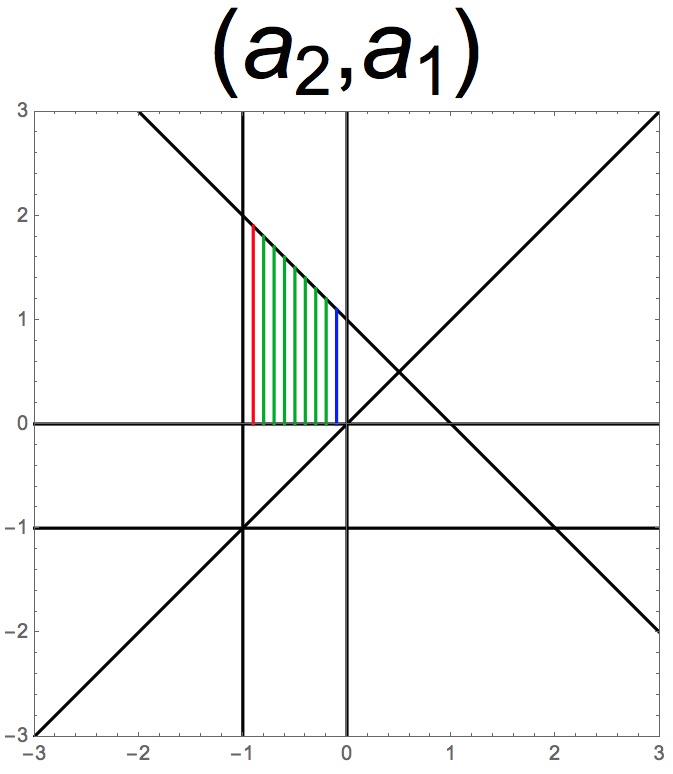}
\includegraphics[height=3.0cm,keepaspectratio=true]{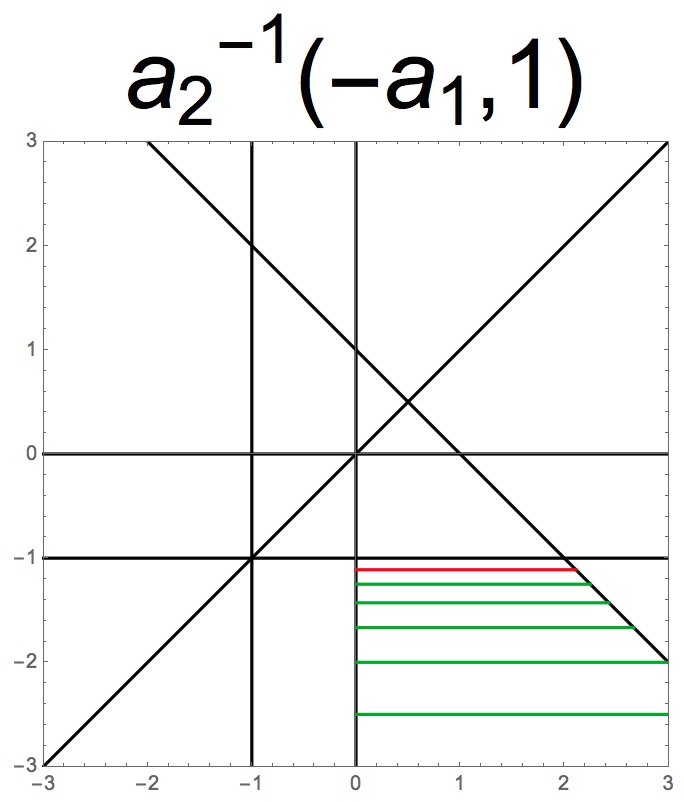}
\includegraphics[height=3.0cm,keepaspectratio=true]{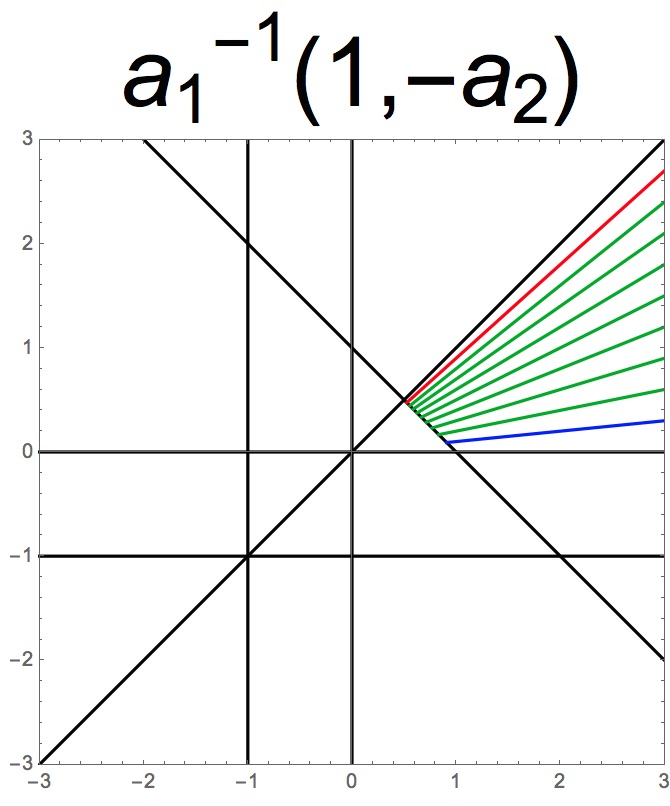}\ 
\includegraphics[height=3.0cm,keepaspectratio=true]{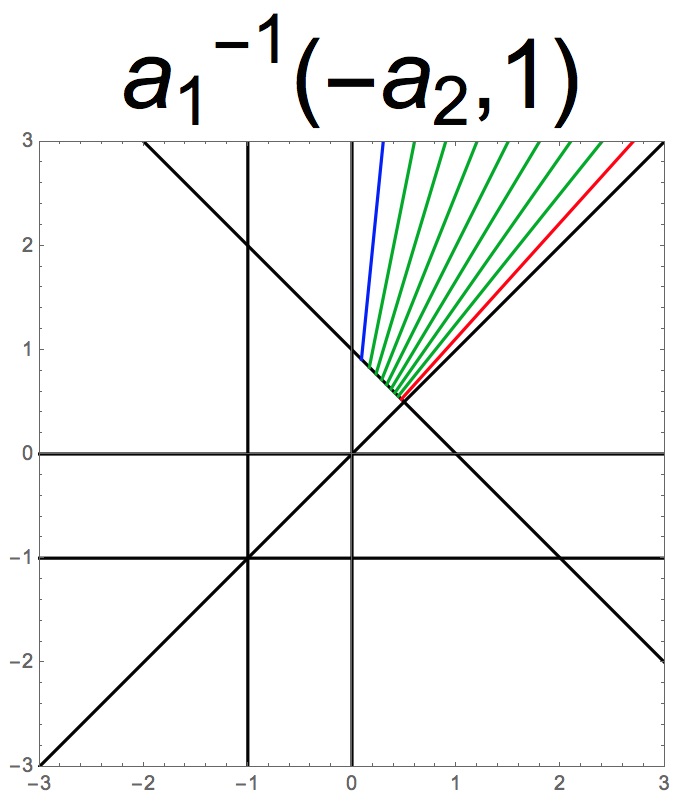}\ 
\includegraphics[height=3.0cm,keepaspectratio=true]{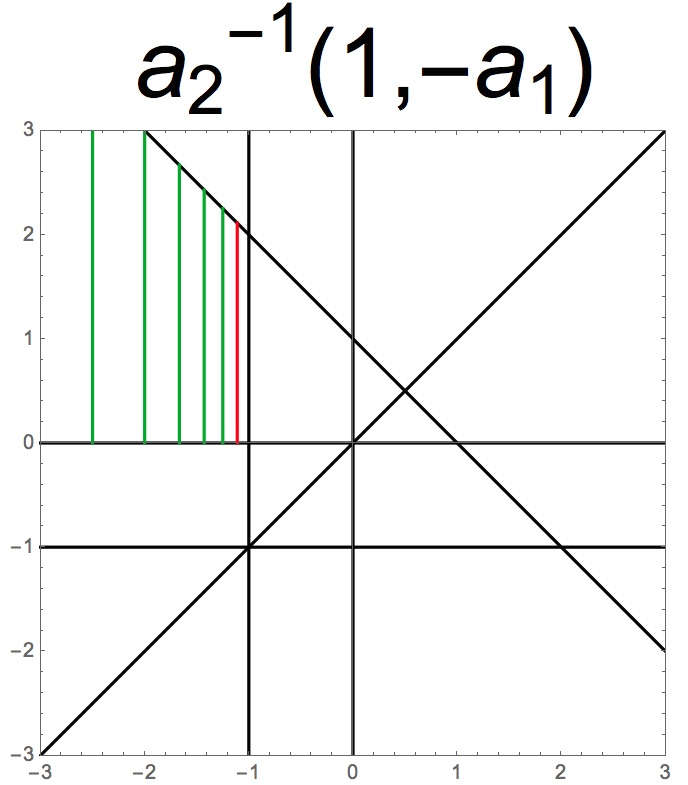}\ 
\includegraphics[height=3.0cm,keepaspectratio=true]{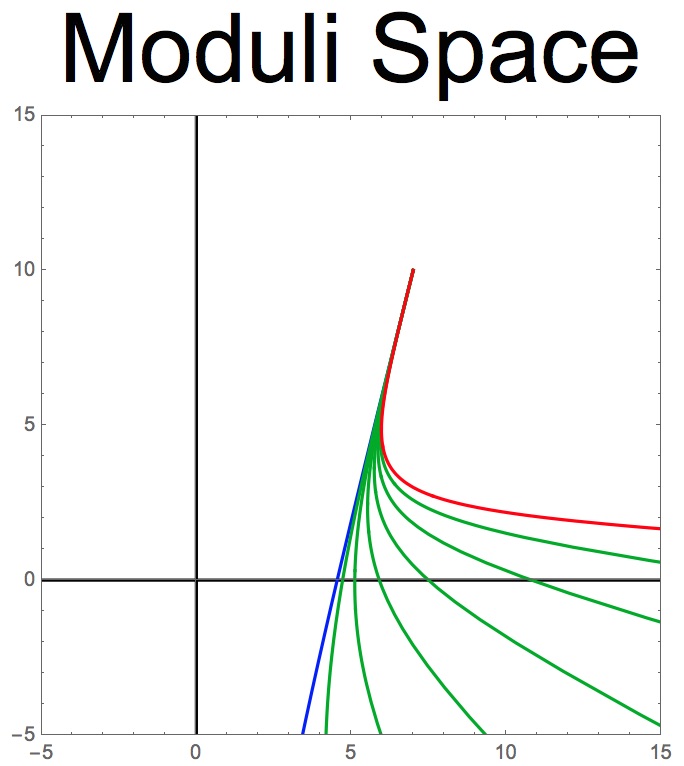}
\end{figure}

\vglue -.3cm\noindent{\bf Case 1c:} The Ricci tensor is positive definite.
 A fundamental region is the triangle with vertices at $\{(0,0),(1,0),(\frac12,\frac12)\}$;
the boundary segment $(t,t)$ for $0<t<\frac12$ belongs to the fundamental region, but the boundary segments $(t,0)$ for $0\le t\le 1$
and $(t,1-t)$ for $\frac12\le t\le 1$ do not lie in the fundamental region. There are portions of the moduli space where the Ricci tensor is
positive definite not present in this fundamental region. The image of the triangle in the moduli space is a bit difficult to picture. The moduli space $\rho>0$ lies to the
right of the curve $ \sigma_r$. There is an exceptional ray $(7,10)+t(1,4)$ for $t\ge0$
which lies to the right of the curve $ \sigma_r$ and which is tangent to this curve at $(7,10)$. The affine
structures with three real exponentials and $\rho>0$ lies to the right of $ \sigma_r$ and to the left of  exceptional ray; these bounding
curves are marked in gray in the moduli space.

In the final two pictures,  $ \sigma_r$ is in red; it is the image
of the line $(t,t)$ for $0<t\le\frac12$. The exceptional ray is marked in blue; it is the boundary
$(t,0)$ for $0<t<1$ and does not belong to this part of the moduli space; it is obtained by the
structures where $\mathcal{Q}$ contains a polynomial as will be discussed later. The final bounding
segment of the triangle is marked in gray; it is the segment $((1+t)/2,(1-t)/2)$ for $0\le t\le 1$;
it lies on the line $a_1+a_2=1$ and has no geometric significance. We refer to Figure~\ref{Fig-5}.\vglue -.3cm
\begin{figure}[H]
\caption{Positive Ricci tensor.}\label{Fig-5}
\vglue -.3cm
\includegraphics[height=2.5cm,keepaspectratio=true]{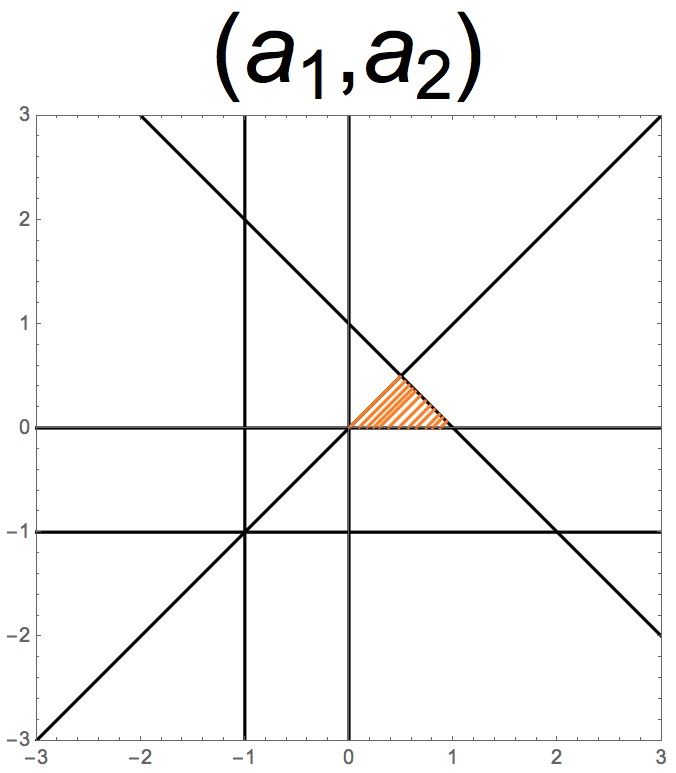}\  
\includegraphics[height=2.5cm,keepaspectratio=true]{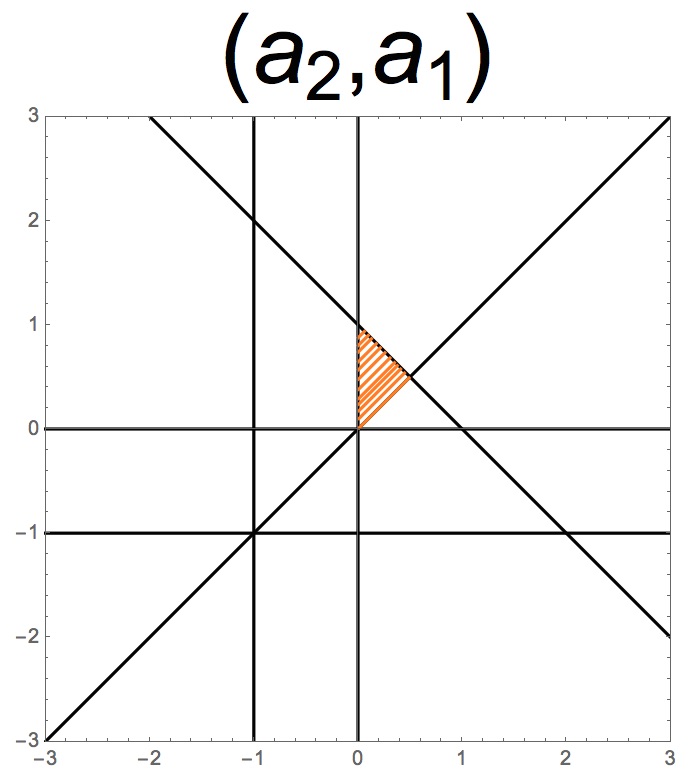}\ 
\includegraphics[height=2.5cm,keepaspectratio=true]{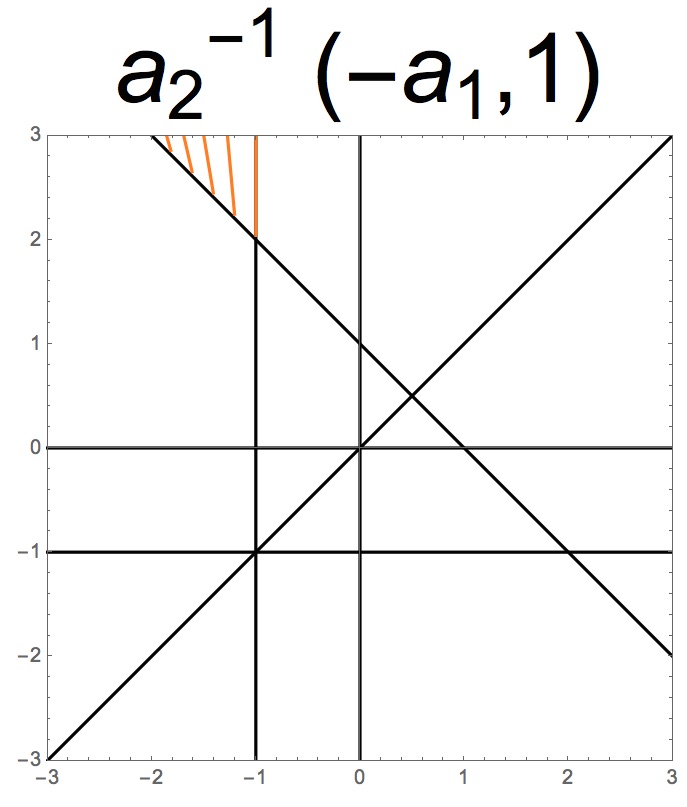}\ 
\includegraphics[height=2.5cm,keepaspectratio=true]{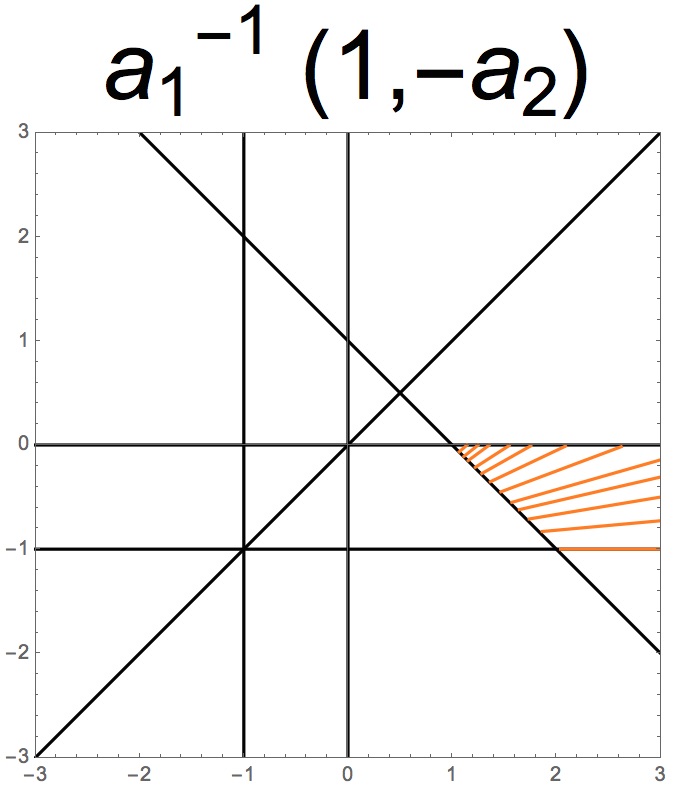}\ 
\includegraphics[height=2.5cm,keepaspectratio=true]{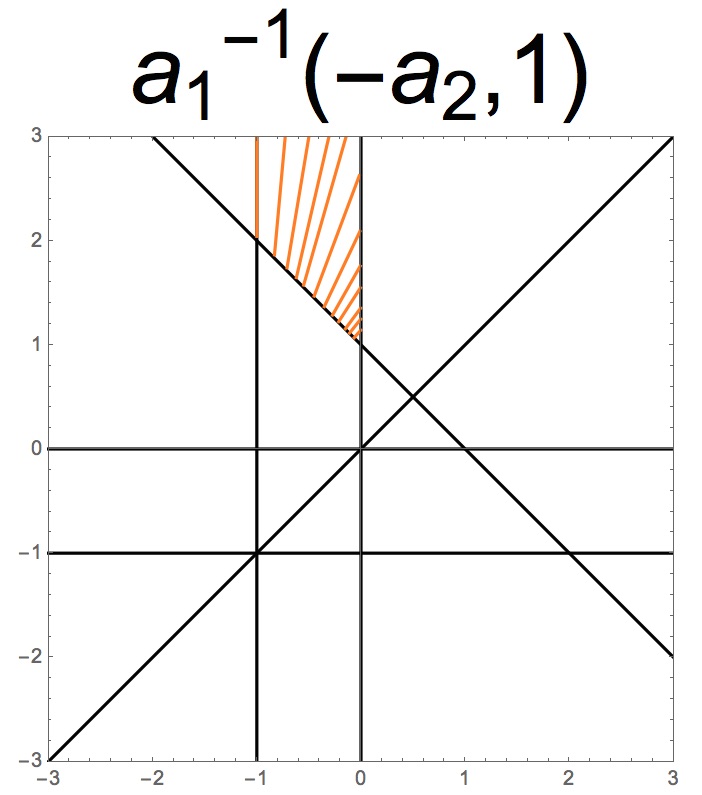}\ 
\includegraphics[height=2.5cm,keepaspectratio=true]{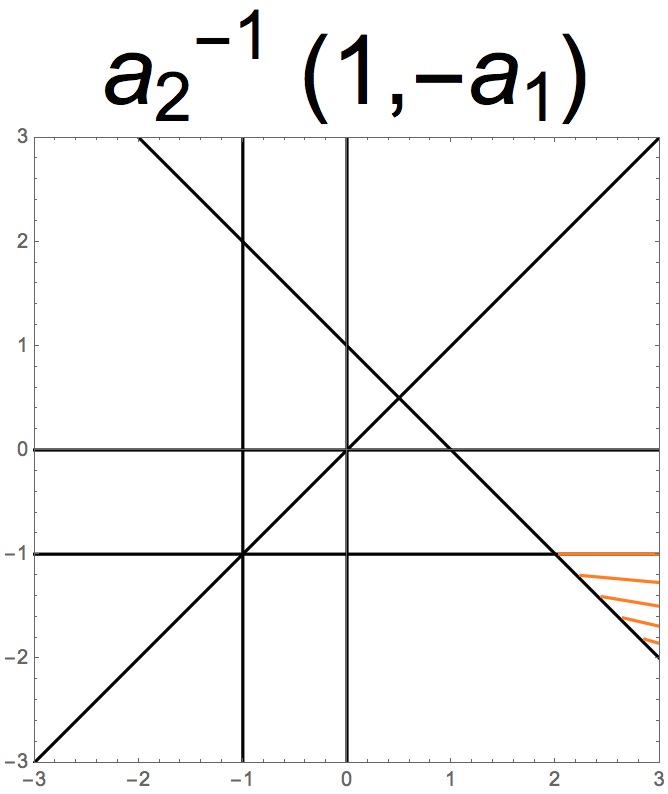}\ 
\includegraphics[height=2.5cm,keepaspectratio=true]{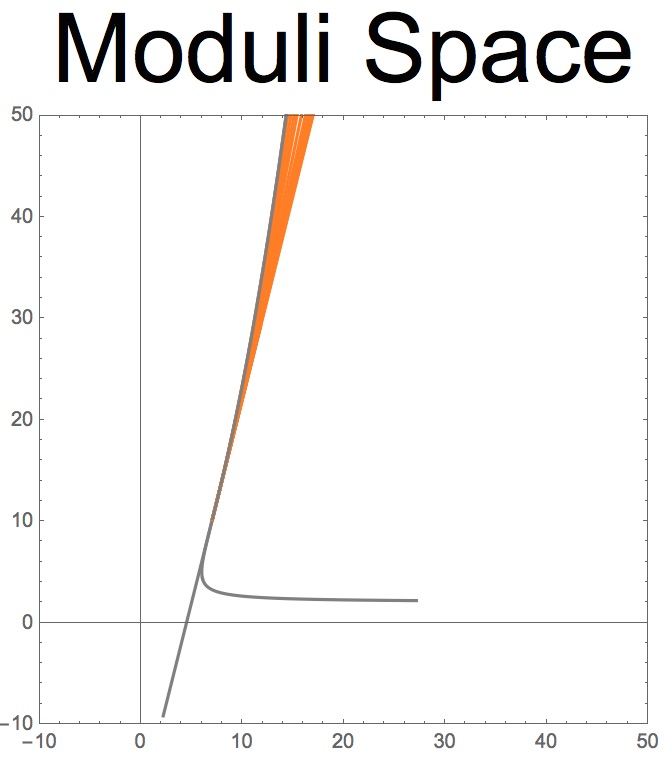} 
\includegraphics[height=2.5cm,keepaspectratio=true]{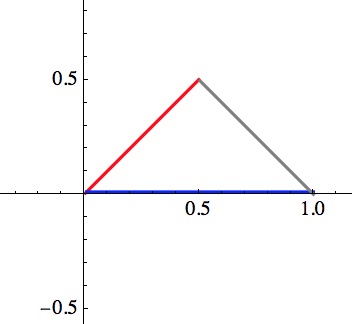}\ 
\includegraphics[height=2.5cm,keepaspectratio=true]{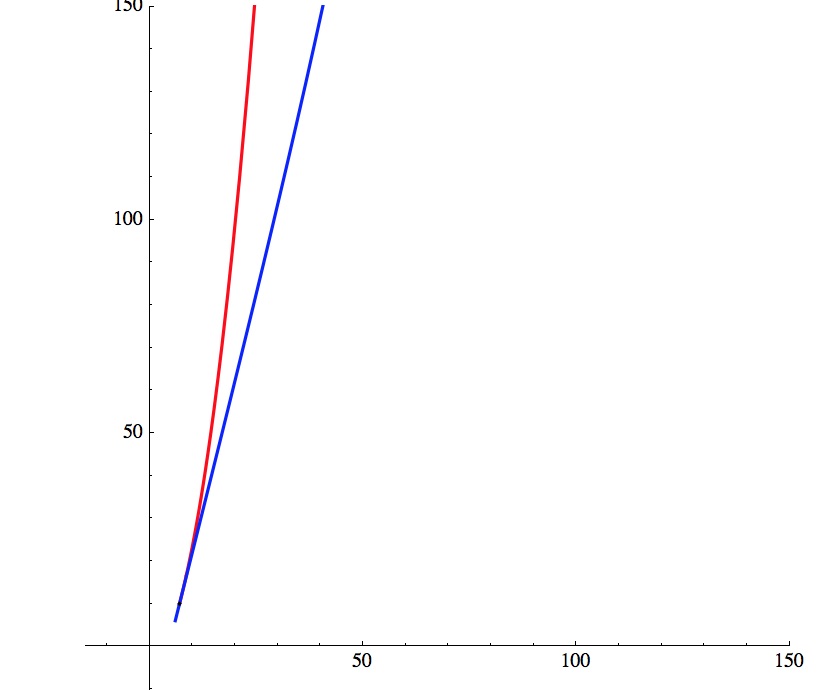}
\end{figure}
\vglue -.3cm\noindent{\bf Case 2:} Linear equivalence if
 {${\mathcal{Q}(\mathcal{M})=\operatorname{Span}\{e^{L_1}\cos(L_2),e^{L_1}\sin(L_2),e^{L_3}\}}$}.
 We set $\Gamma=\Gamma_c^2(b_1,b_2)$ where $b_1\ne1$ and $(b_1,b_2)\ne(0,0)$.
 We have $b_1>1$ corresponds to $\rho$ positive definite
and $b_1<1$ corresponds to $\rho$ indefinite; $(b_1,b_2)$ and $(\tilde b_1,\tilde b_2)$ are linearly equivalent if and only if
$b_1=\tilde b_1$ and $b_2=\pm\tilde b_2$. The two fundamental domains and
the images in the moduli spaces are shown in Figure~\ref{Fig-6}.\vglue -.4cm\par\noindent
\begin{figure}[H]
\caption{Complex exponentials.}\label{Fig-6}
\vglue -.3cm\includegraphics[height=2.2cm,keepaspectratio=true]{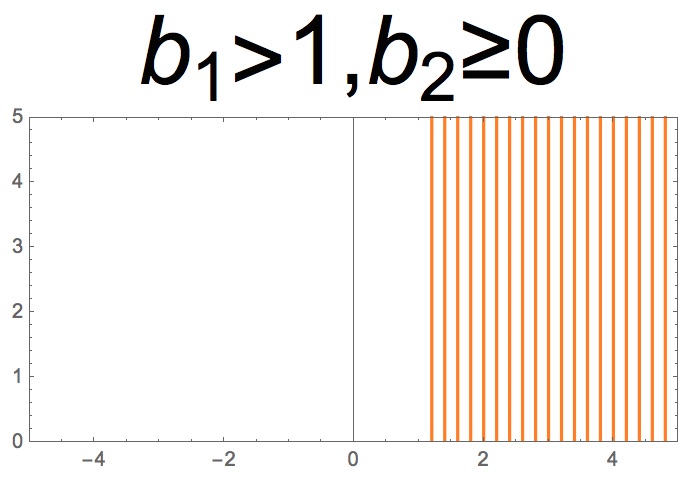}\ 
\includegraphics[height=2.2cm,keepaspectratio=true]{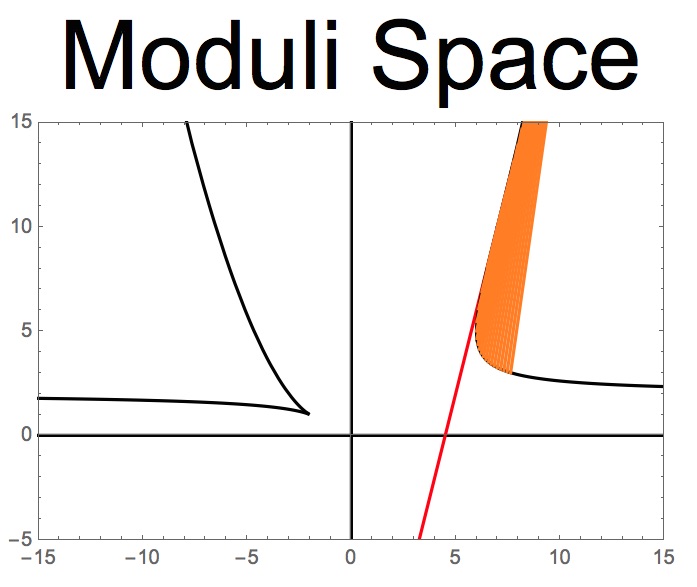}\ 
\includegraphics[height=2.2cm,keepaspectratio=true]{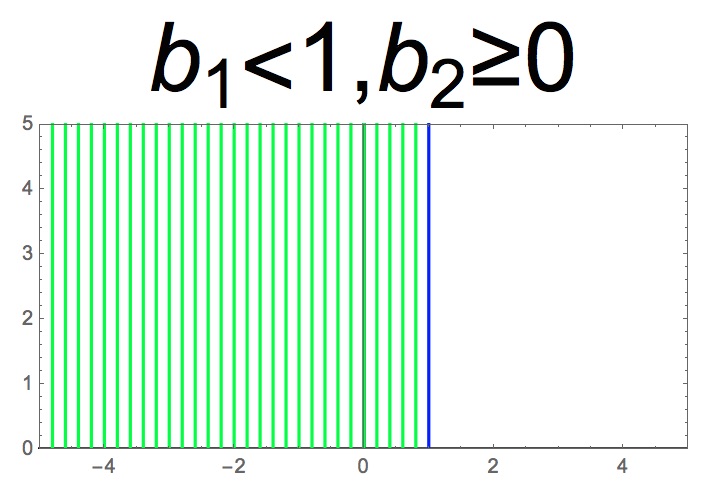}\ 
\includegraphics[height=2.2cm,keepaspectratio=true]{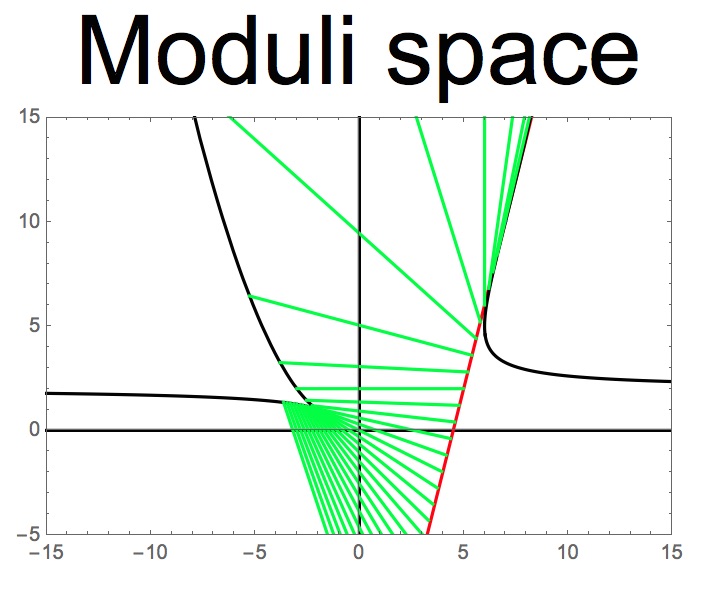}
\end{figure}
\vglue -.3cm\noindent{\bf Case 3:}  ${\mathcal{Q}}$ involves non-trivial polynomials. We have
$\Gamma=\Gamma_p^2(a)$ for $a\ne0$ or $\Gamma=\Gamma_q^2(\pm1)$. If $a>0$, then $\rho$ is positive definite
and $(\psi,\Psi)(\Gamma_p^2)=(7,10)+\frac1a(1,4)$. And $\rho$ is positive for $\Gamma=\Gamma_q^2( +1)$ and we have
$(\psi,\Psi)(\Gamma_q^2(+1))=(7,10)$. These two structures give rise to the closed ray $(7,10)+t(1,4)$ for $t\ge0$ 
marked in red in Figure~\ref{Fig-7}.
Similarly, if $a<0$, then $\rho$ is negative definite; this structure together with $\Gamma_q^2(-1)$ give rise to the closed
ray $(7,10)-t(1,4)$ for $t\ge0$ in the moduli space marked in blue in Figure~\ref{Fig-7} below. These two rays divide the portion of the moduli space where $\mathcal{Q}$
involves 3 real exponentials (Case 1) from the portion of the moduli space where $\mathcal{Q}$ contains complex exponentials (Case 2).
We refer to Figure~\ref{Fig-7}.
\vglue -.3cm\begin{figure}[H]
\caption{The exceptional line $\mathcal{Q}$ involves polynomials.}\label{Fig-7}
\vglue -.3cm\includegraphics[height=3.0cm,keepaspectratio=true]{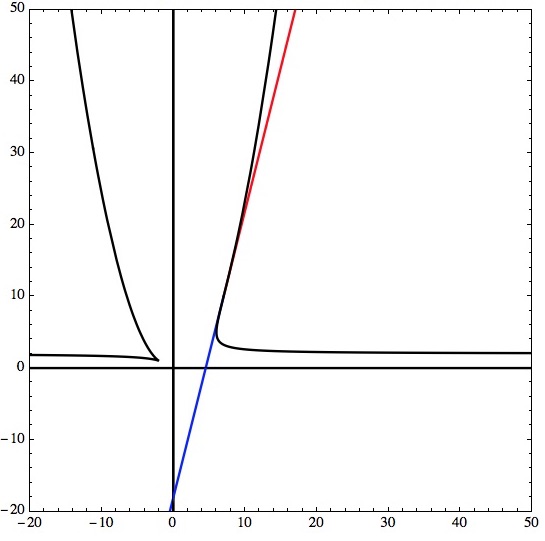}\vglue-.3cm
\end{figure}

\end{document}